\documentclass{amsart}
\pdfoutput=1

\usepackage[
  top=1.3in,
  bottom=1.3in,
  left=1.25in,
  right=1.25in
]{geometry}

\usepackage{amsmath,amsfonts,amssymb}
\usepackage[numbers]{natbib}
\usepackage{graphicx}
\usepackage{enumerate}
\usepackage{mathtools}
\usepackage{subcaption}
\usepackage{xcolor}
\usepackage{comment}
\usepackage{tikz}
\usepackage{bbm}
\usepackage{mathrsfs}
\usepackage[hidelinks]{hyperref}
\usepackage{overpic}

\newtheorem{thm}{Theorem}[section]
\newtheorem{lm}[thm]{Lemma}
\newtheorem{defn}[thm]{Definition}

\newtheorem{prop}[thm]{Proposition}

\newtheorem{cor}[thm]{Corollary}

\newtheorem{ex}[thm]{Example}

\newtheorem{rmk}[thm]{Remark}
\newtheorem{alg}[thm]{Algorithm}

\numberwithin{equation}{section}

\usepackage[obeyFinal,textsize=footnotesize]{todonotes}

\newcommand{\E}{\mathbb{E}}
\newcommand{\pr}{\mathbb{P}}

\newcommand{\R}{\mathbb{R}}

\newcommand{\la}{\langle}
\newcommand{\ra}{\rangle}
\newcommand*\dif{\mathop{}\!\mathrm{d}}

\newcommand{\rvline}{\hspace*{-\arraycolsep}\vline\hspace*{-\arraycolsep}}

\title[Tridiagonal Model for $\beta$-Dyson Brownian Motion]{On the Limit of the Tridiagonal Model \\for $\beta$-Dyson Brownian Motion}

\author{Alan Edelman}
\address{Department of Mathematics and Computer Science \& AI Laboratory, Massachusetts Institute of Technology, Cambridge, Massachusetts}
\email{edelman@mit.edu}
\author{Sungwoo Jeong}
\address{Department of Mathematics, Cornell University, Ithaca, New York}
\email{sjeong@cornell.edu}
\author{Ron Nissim} 
\address{Department of Mathematics, Massachusetts Institute of Technology, Cambridge, Massachusetts}
\email{rnissim@mit.edu}

\begin{document}

\begin{abstract} 
    In previous work, a description of the result of applying the Householder tridiagonalization algorithm to a G$\beta$E random matrix is provided by Edelman and Dumitriu. The resulting tridiagonal ensemble makes sense for all $\beta>0$, and has spectrum given by the $\beta$-ensemble for all $\beta>0$. Moreover, the tridiagonal model has useful stochastic operator limits which was introduced and analyzed in subsequent studies. 
    \par In this work, we analogously study the result of applying the Householder tridiagonalization algorithm to a G$\beta$E process which has eigenvalues governed by $\beta$-Dyson Brownian motion. We propose an explicit limit of the upper left $k \times k$ minor of the $n \times n$ tridiagonal process as $n \to \infty$ and $k$ remains fixed. We prove the result for $\beta=1$, and also provide numerical evidence for $\beta=1,2,4$. This result provides a speculation for the form of a dynamical $\beta$-stochastic Airy operator with smallest $k$ eigenvalues evolving according to the $n \to \infty$ limit of the largest, centered and re-scaled, $k$ eigenvalues of $\beta$-Dyson Brownian motion. Unfortunately, additional numerical simulations and a perturbative calculation which we provide do not substantiate this speculation.
\end{abstract}

\maketitle






\section{Introduction}

Dyson Brownian motion (DBM) introduced by Freeman Dyson in 1962 \cite{Dy62} is a widely studied stochastic model for a system of particles that possess pairwise repulsive behavior. For any $\beta >0$, $\beta$-Dyson Brownian motion ($\beta$-DBM) is the dynamics of a system $\{\lambda_i\}_{i=1}^{n}$ undergoing the following system of stochastic differential equations
\begin{equation}\label{Dyson BM}
d\lambda_i =\left(\sum_{1\leq j \neq i \leq n} \frac{1}{\lambda_i - \lambda_j}-\sum_{i=1}^{n}\frac{\lambda_i}{2}\right) dt  + \sqrt{\frac{2}{\beta}} d B_j.
\end{equation}
As long as  $\lambda_1(0),\dots ,\lambda_n(0)$, are distinct, it is known that \eqref{Dyson BM} has a unique strong solution and $\lambda_1(t),\dots ,\lambda_n(t)$ are a.s.\  distinct for all $t>0$ as long as $\beta \geq 1$. Moreover $\beta$-DBM has a unique invariant distribution given by the $\beta$-ensemble (where $\mathcal{Z}$ is for normalization)
\begin{equation}\label{betaEigenvalDist}
     \mathcal{Z}^{-1}\prod_{i<j}|\lambda_i -\lambda_j|^{\beta}\prod_{i=1}^n e^{-\lambda_i^2/4} dx_i.
\end{equation}

In the random matrix context, for $\beta=1,2,4$, $\beta$-DBM describes the evolution of the eigenvalues of matrices evolving according to the G$\beta$E process, which is the following matrix valued Ornstein--Uhlenbeck process 
\begin{equation}\label{G beta E Process}
    dM = -M dt +\sqrt{2}d B_{G\beta E},
\end{equation}
that leaves the G$\beta$E distribution invariant \cite{Anderson_Guionnet_Zeitouni_2009,forrester2010log}. And as a consequence G$\beta$E matrices (alternatively known as $\beta$-Hermite ensembles) have eigenvalue distribution given by \eqref{betaEigenvalDist}.

For $\beta=1,2,4$ the G$\beta$E matrices coincide with the highly studied Gaussian  Orthogonal Ensemble (GOE), Gaussian Unitary Ensemble (GUE), and Gaussian Symplectic Ensemble (GSE) respectively. These three G$\beta$Es are described by the canonical random symmetric, Hermitian, and self-dual matrix models, respectively, endowed with Gaussian entries. For general $\beta>0$, there is no immediate generalization.


However, by applying the Householder tridiagonalization used in numerical linear algebra \cite{bauer1959sequential,wilkinson1962householder} on G$\beta$E matrices, one obtains the following symmetric tridiagonal model \cite{DuEd02}:
\begin{equation}\label{StationaryTridiagonal}
    \begin{pmatrix}
        N(0,2) & \chi_{\beta (n-1)} &  &  &   &  &  \\
        \chi_{\beta (n-1)} &  N(0,2) & \chi_{\beta (n-2)} &  & &  & \\
         &  \chi_{\beta(n-2)}& \ddots & \ddots &  &  &  \\
         & & \ddots & \ddots & \ddots & & \\
         & & & \ddots & N(0, 2) & \chi_{2\beta} &\\ 
         &  &  &  & \chi_{2\beta} & N(0,2) & \chi_{\beta}\\
          &  &  &  &  & \chi_{\beta} & N(0,2)
    \end{pmatrix}.
\end{equation}
It was verified \cite{DuEd02} that for all $\beta >0$ one has \eqref{betaEigenvalDist} as the eigenvalue distribution for the tridiagonal matrix ensemble \eqref{StationaryTridiagonal}.  In addition to providing a matrix model for the the $\beta$ ensemble for all $\beta$, one can find stochastic differential operator limits ($n\to\infty$) for the tridiagonal matrix model which provide further insight into quantities such as the largest eigenvalue governed by the $\beta$-Tracy Widom distribution. Such an operator limits are introduced and analyzed in \cite{Ed04,EdSu07}, and further analyzed in \cite{ramirez2011beta}. Over the past two decades the G$\beta$E tridiagonal model and its operator limits have been studied and applied extensively, e.g., see \cite{dumitriu2006global, li2010beta,dumitriu2012global,AlGu13,allez2014tracy,krishnapur2016universality,valko2017sine,HoPa17,gorin2018stochastic,LambertPaquette03,paquette2023universality,trogdon2024computing}, and related work on tridiagonalization of random matrices \cite{deift2021conjugate,ding2022conjugate}.
\begin{figure}[h]
    \begin{overpic}[width=0.95\textwidth]{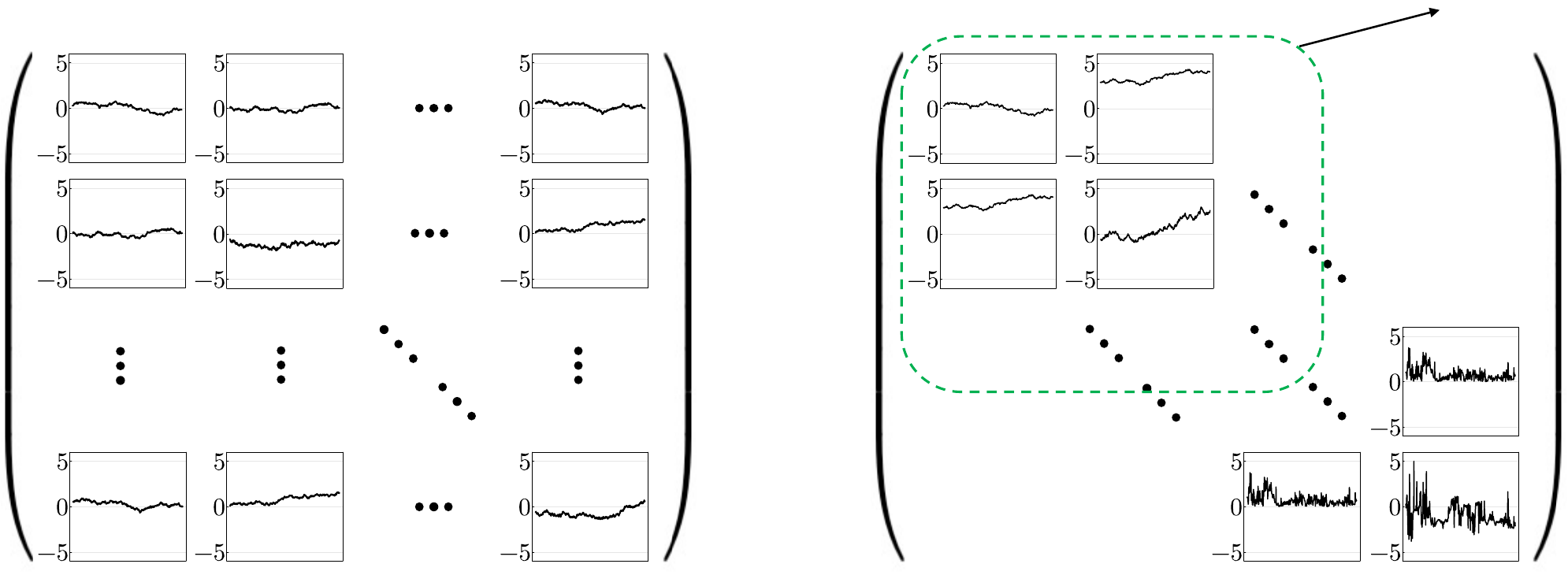}
    \put(82.5, 34.4) {\rotatebox{16}{$n$ large}}
    \put(92, 35.3) {\eqref{Tridiagonal Evolution Informal}}
    \put(48, 16.5) {$\longrightarrow$}
    \end{overpic}    
    \caption{Householder Tridiagonalization (right) of a $10\times 10$ GOE process (left)}\label{fig:intropic}
\end{figure}

In this work, we consider applying the Householder tridiagonalization procedure to the $n\times n$ G$\beta$E process, for each $t$, as depicted graphically in Figure~\ref{fig:intropic}. The resulting process is a real symmetric tridiagonal matrix valued stochastic process with eigenvalue distribution governed by $\beta$-DBM, \eqref{Dyson BM}. Since these tridiagonal matrices are symmetric, we can treat the entries as a system of $2n-1$ scalar stochastic processes. 

Although at a fixed time these processes are described simply as in \eqref{StationaryTridiagonal}, as processes they appear to be considerably more complicated. For instance, our numerical results in Section~\ref{sec:notgaussian} indicate that the diagonal entries are not Gaussian processes despite their Gaussianity at a fixed time. Nevertheless, as the dimension of the matrices $n \to \infty$, a fixed $k \times k$ upper-left corner submatrix process can be roughly described as below where the entries are independent to each other: 
\small
\begin{equation}\label{Tridiagonal Evolution Informal}
    \begin{pmatrix}
        \sqrt{2}\mathrm{OU}(1) & \mathrm{ROU}_{\beta(n-1)}(1) &       \\
       \mathrm{ROU}_{\beta(n-1)}(1) & \sqrt{2}\mathrm{OU}(3) &  \hspace{-2cm} \ddots \\
         &  \ddots  &  \hspace{-3cm}\ddots &  \hspace{-7cm}\ddots    \\
         & \hspace{1.5cm} \ddots & \hspace{-1.5cm}\ddots\\
          &   & & \mathrm{ROU}_{\beta(n-k+1)}(k-1)\\
            & & \hspace{-1cm}\mathrm{ROU}_{\beta(n-k+1)}(k-1) & \sqrt{2}\mathrm{OU}(2k-1)
    \end{pmatrix}.
\end{equation}
\normalsize
($\mathrm{OU}(m)$ denotes a stationary Ornstein--Uhlenbeck process with parameters $\theta=m, \sigma = \sqrt{2m}$, and $\mathrm{ROU}_{d}(m)$ is a radial Ornstein--Uhlenbeck process with dimension parameter $d$, and $\theta=m, \sigma = \sqrt{2m}$; refer to Section \ref{OU Section} for the notation.) Furthermore, one can check from the SDE governing $\mathrm{ROU}_{d}(m)$, that $\frac{1}{\sqrt{2d}}(\mathrm{ROU}_{d}(m)^2-d) \to \mathrm{OU}(2m)$ so our main result will use this re-centering and scaling.

Our main result, Theorem~\ref{Main Theorem Statement}, states that for $\beta=1$, the tridiagonalization of the $n\times n$ G$\beta$E process, denoting diagonal entries by $a_j(t)$ and off-diagonal entries by $b_j(t)$, as $n\to\infty$
\begin{align*}
    a_j(t) &\to \sqrt{2}\mathrm{OU}(2j-1),\\
    \frac{1}{\sqrt{\beta n}}(b_j(t)^2 - \beta n) &\to \sqrt{2}\mathrm{OU}(2j),
\end{align*}
for entries in the $k \times k$ upper-left corner submatrix process, i.e., $j\leq k$. Moreover all of these limiting Ornstein--Uhlenbeck processes are independent of each other. Although we only prove the statement for $\beta=1$, the proof should easily generalize to the $\beta = 2,4$ cases, and our numerical evidences substantiate this claim.  

This result suggests a simple tridiagonal model which has the same largest few eigenvalues as $\beta$-DBM in the $n \to \infty$ limit, (namely extending the pattern \eqref{Tridiagonal Evolution Informal} all the way to $k=n$) and this tridiagonal model should converge to a time evolving stochastic Airy operator defined in Section \ref{open problem section}. Unfortunately our result only concern a finite number of entry processes, so we cannot actually recover eigenvalue or operator limit information. Moreover, our numerical experiments and a perturbative calculation provide evidence against this limit. In Section~\ref{open problem section} we further discuss conjectures and further open problems related to tridiagonal processes and their eigenvalues. 

In addition to proving the main theorem, we perform numerical experiments to support our result. Since we study large matrix processes our simulations are implemented on parallel machines to increase efficiency. See Section~\ref{Numerics} for details.



\subsection{Organization of the Paper}

Section~\ref{Notation and Background} covers the notation and background needed to read the rest of the paper. In particular Sections~\ref{OU Section} and \ref{Householder and Tridiag Section} provide the notation and setup needed to understand the main result, numerical experiments, and future directions, while Sections~\ref{Chebyshev Polynomial Section}, \ref{Non-crossing Partitions}, \ref{Factorization Section},  and \ref{Approximate thetas}, are only referred to during the proof of Theorem \ref{Main Theorem Statement}, which is the main result.

In Section~\ref{sec:mainresult} we state the main result, Theorem~\ref{Main Theorem Statement}, and outline its proof. Before providing a detailed proof of Theorem~\ref{Main Theorem Statement}, we provide numerical results in Section~\ref{Numerics}, and remarks on prior research in the areas, as well as potential future direction in Section~\ref{open problem section}. These two sections can be read independently of the proof of Theorem~\ref{Main Theorem Statement} which is deferred to Sections~\ref{Approximation Section} and \ref{Moment Method Section}, as well as the Appendix which supplements Section~\ref{Approximation Section}.

\section{Notation and Background}\label{Notation and Background}
 
\subsection{GOE, Ornstein--Uhlenbeck Process, and Dyson Brownian Motion}\label{OU Section}
For each $n \in \mathbb{N}$, let $m^{(n)}(t)$ denote an $n \times n$ GOE Ornstein--Uhlenbeck process. That is, $\{(m^{(n)}(0))_{i,j}\}_{1 \leq i \leq j \leq n}$ are all independent, $(m^{(n)}(0))_{i,j}=(m^{(n)}(0))_{j,i}$ for all $i,j \in [n]$, and
\begin{equation}
    \begin{cases}
        (m^{(n)}(0))_{i,j} \sim N(0,1) \hspace{3mm}  \text{if } i\neq j\\
        (m^{(n)}(0))_{i,j} \sim N(0,\sqrt{2}) \hspace{3mm} \text{if } i= j
    \end{cases}
\end{equation}
\noindent
and $m^{(n)}(t)$ evolves according to 

\begin{equation}\label{GOE OU SDE}
    \dif{m^{(n)}}(t) = -m^{(n)}(t) \dif{t} +  \sqrt{2}\dif{B^{(n)}_{\mathrm{GOE}}}(t)
\end{equation}
where $B^{(n)}_{\mathrm{GOE}}(t):= \frac{1}{\sqrt{2}}(B^{(n)}(t)+B^{(n)}(t)^T)$ and $B^{(n)}(t)$ is an $n \times n$ matrix with independent standard Brownian Motion entries.
It is well known that this SDE is globally well-posed and the unique solution $m^{(n)}(t)$ is a stationary matrix valued Gaussian process with covariances given by
\begin{equation}\label{GOE OU Covariances}
\E[(m^{(n)}(t))_{i_1,j_1}(m^{(n)}(t))_{i_2,j_2}] = \delta_{\{i_1,j_1\},\{i_2,j_2\}}(1+\delta_{i_1,j_1})e^{-|t-s|}.
\end{equation}
For each $n \in \mathbb{N}$ we take a (continuous) copy of $m^{(n)}(t)$ which will all be processes living in a single probability space $(\Omega, \mathcal{F}, \pr)$.

Let us also define for more general choice of parameters, the stationary Ornstein--Uhlenbeck process.
\begin{defn}\label{Ornstein-Uhlenbeck Process}
    The normalized stationary scalar Ornstein--Uhlenbeck process $\mathrm{OU}(c)$, is the unique solution to
    \begin{equation}\label{Stationary Scalar Ornstein Uhlenbeck process SDE}
    \begin{split}
        &dx(t)=-m x(t)\dif{t} +\sqrt{2m} \dif{W(t)} \\
        &x(0) \sim N(0,1)
    \end{split}
\end{equation}
for a standard one dimensional Brownian motion $W(t)$. Or equivalently, $\mathrm{OU}(m)$ is a stationary Gaussian process for all time with covariances given by $\E[x(t)x(s)]=e^{-m|t-s|}$. Furthermore if $m=(m_1,...,m_n)$ denote by $\mathbf{OU}_n(m)$ the vector valued Ornstein--Uhlenbeck process that has independent $\mathrm{OU}(m_i)$ entries. If $c$ is scalar, the notation $\mathbf{OU}_n(m)$ will be short for $\mathbf{OU}_n((m,...,m))$.
\end{defn}

Lastly we define the stationary radial Ornstein--Uhlenbeck process.

\begin{defn}
    The normalized stationary radial Ornstein--Uhlenbeck process $\mathrm{ROU}_d(m)$, is the unique solution to 
    \begin{equation}\label{Stationary Radial Ornstein Uhlenbeck process SDE}
    \begin{split}
        &dR(t)=\bigg(\frac{m(d-1)}{R(t)}-mR(t)\bigg)\dif{t} +m \dif{W(t)} \\
        &R(0) \sim \chi_d
    \end{split}
\end{equation}
for a standard one dimensional Brownian motion $W(t)$.
\end{defn}

\begin{rmk}
    In the case that $d \in \mathbb{N}$, $\mathrm{ROU}_d(m)=||\mathbf{OU}_d(m)||$ in distribution.
\end{rmk}

\subsection{Householder Matrices and Tridiagonalization}\label{Householder and Tridiag Section}

\begin{defn}[Householder Reflector]\label{def:householder}
For a given unit vector $v\in \mathbb{R}^n$, an $n\times n$ matrix $H = I - 2vv^T$ is called a Householder reflector. 
\end{defn}

\begin{rmk}
The matrix $H$ geometrically represents a flip about the plane orthogonal to $v$. For any given vector $x\in \mathbb{R}^n$, letting $v = \frac{x - \|x\| e_1}{\|x - \|x\|e_1\|}$ we obtain a Householder reflector $H$ by Definition~\ref{def:householder}. Then we have
\begin{equation*}
    H x = \begin{pmatrix} \|x\| \\ 0 \\ \vdots \\ 0\end{pmatrix}.
\end{equation*}
Moreover any Householder reflector $H$ is both symmetric and orthogonal. 
\end{rmk}

Next we recall the well-known Householder matrix based tridiagonalization procedure which starts with a symmetric matrix $A$ and outputs a symmetric tridiagonal matrix $T$ which is similar to $A$. 

\begin{alg}[Householder Tridiagonalization]\label{alg:householder}
\,
    \begin{enumerate}
        \item Given a symmetric $n \times n$ matrix $A$, set $A=A_0$ and $k=0$
        \item Write the $(n-k) \times (n-k)$ matrix $A_k$ in the block form
        \begin{equation}
            A_k=
            \begin{pmatrix}
                a_k
                &\rvline & x_k^T \\
                \hline
                x_k & \rvline &
                \Tilde{A}_k
            \end{pmatrix}
        \end{equation}
        where $a_k$ is a $1 \times 1$, $x_k$ is $(n-k-1)\times 1$ and $\Tilde{A}_k$ is $(n-k-1) \times (n-k-1)$

        \item Take $v_k := \frac{x_k-\|x_k\|e_{k+2}}{\|x_k-\|x_k\|e_{k+2}\|}$, where $x_k$ is a vector in $\R^n$ by taking the first $k+1$ entries to be $0$, and define the Householder matrix $H_k=I_{n}-2v_kv_k^T$. Note that $H_k$ can be written in the following block form,
        \begin{equation}
            H_k:=
            \begin{pmatrix}
                1
                & \rvline & 0 \\
                \hline
                0 & \rvline &
                \Tilde{H}_k
            \end{pmatrix},
        \end{equation}
        where $\Tilde{H}_k$ is an $(n-k-1) \times (n-k-1)$ matrix.
        \item Preform the similarity transformation

        \begin{equation}
            \begin{split}
                H_k^T A_k H_k &= 
                \begin{pmatrix}
                a_k
                & \rvline & \|x_k\| & 0 \\
                \hline
                \begin{matrix}
                    \|x_k\| \\ 0
                \end{matrix} & \rvline &
                \Tilde{H}_k^T \Tilde{A}_k \Tilde{H}_k
            \end{pmatrix}\\
            &=: \begin{pmatrix}
                a_k
                & \rvline & b_k & 0 \\
                \hline
                \begin{matrix}
                    b_k \\ 0
                \end{matrix} & \rvline &
                A_{k+1}
            \end{pmatrix}
            \end{split}
        \end{equation}
        \item If: $k< n-2$, set $k \to k+1$ and  go back to Step 2.\\
        Else: Return a tridiagonal matrix 
        \begin{equation}\label{eq:tridiagonalmatrix}
            T :=
            \begin{pmatrix}
                a_1 & b_1 & 0 &\cdots & 0 \\
                b_1 & a_2 & b_2 &  & \vdots \\
                0 & b_2 & a_3 & \ddots & 0  \\
                \vdots &  & \ddots & \ddots& b_{n-1} \\
                0 & \cdots  & 0 &  b_{n-1} & a_n
            \end{pmatrix}
        \end{equation}
    \end{enumerate}
    
\end{alg}    

\begin{rmk}
    In numerical linear algebra a standard practice in Step 3 is to select $v_k$ as a unit vector parallel to $x_k + \text{sign}((x_k)_1)||x_k||e_1$ to ensure numerical stability of the algorithm. However in this work we simply use the Householder reflector that gives $b_k > 0$ for all $k$. 
\end{rmk}

\begin{rmk}\label{rmk:DE02-reminder}
    We remind the reader that if one applies Algorithm \ref{alg:householder} to an $n \times n$ $\mathrm{GOE}$ distributed random matrix, then at every step $k$, $x_k$ is a standard Gaussian vector in $\R^{n-k-1}$, and the $x_k$ are all independent of each other. Moreover, the final result of the algorithm, $T$, is distributed as in equation \eqref{StationaryTridiagonal} where all the entries are independent of each other. We refer the reader to \cite{DuEd02} for additional background.
\end{rmk}

In our setting, replace $n$ with $n+1$, take $M_0(t)=M(t)$ to correspond to $A_0$ in the tridiagonalization algorithm \ref{alg:householder}, and for $k\geq 0$ we take $M_k(t)$ to correspond to the partially tridiagonalized matrix
\begin{equation*}
    \begin{pmatrix}
    \begin{matrix}
        a_1 & b_1 & & \\
        b_1 & a_2 & b_2 & \\
        & & & \ddots & b_{k-2}\\
        & & & b_{k-2} & a_{k-1}
    \end{matrix}
    &\rvline & 0 \\
    \hline
    0 & \rvline & A_k
    \end{pmatrix}.
\end{equation*}
Similarly we take $x_k(t)$ to correspond to $x_k$, and $Q_k(t)$ to correspond to $\tilde{H}_k$. Finally we let $a_k(t)$ and $b_k(t)$ correspond to $a_k$ and $b_k$ respectively, the entries of the resulting tridiagonalization procedure, and and $\mathcal{T}(t)$ for the tridiagonal $T$. We also sometimes write $a^{(n)}_k(t)$ and $b^{(n)}_k(t)$ to emphasize the dependence on $n$. 

However, we diverge slightly from the notation of algorithm \ref{alg:householder}. In particular we will treat the $x_k(t)$ vectors as vectors in $\R^{n+1}$ with the first $k+1$ entries equal to $0$, and we also treat $Q_k(t)$ as an $(n+1) \times (n+1)$ matrix by acting as the identity on vectors which only have their first $k+1$ components nonzero.

We can summarize the steps of the tridiagonalization procedure in the following recursion with $M_0(t) = M(t)$:
\begin{equation}
    \begin{split}
        a_{k+1}(t)&=e_{k+1}^T M_k(t)e_{k+1},\\
        \vartheta_{k}(t)&=\frac{1}{\sqrt{n}}M_k(t)e_{k+1}-\|\vartheta_{k-1}(t)\|e_k-\frac{1}{\sqrt{n}}a_{k+1}(t)e_{k+1},\\
        b_{k+1}(t) &= \|\vartheta_k(t)\|,\\
        Q_k(t)&=I-\frac{2(\vartheta_k(t)-\|\vartheta_k(t)\|e_{k+2})(\vartheta_k(t)-\|\vartheta_k(t)\|e_{k+2})^T}{\|\vartheta_k(t)-\|\vartheta_k(t)\|e_{k+2}\|^2},\\
        M_{k+1}(t) &= Q_k(t)M_k(t)Q_k(t).
    \end{split}
\end{equation}
Where $\vartheta_k(t)$ is defined via the relation, $\vartheta_k(t):=\frac{x_k(t)}{\sqrt{n}}$. Once again we stress that the $\vartheta_k(t) \in \R^{n+1}$ with the first $k+1$ entries identically zero.

The division by zero in the definition of $Q_k(t)$ above might make some readers uncomfortable. But observe, at a fixed time, $x_k(t) \sim \chi_{n+1-k}$ (e.g., see \cite{DuEd02}), so clearly $\vartheta_k(t)-\|\vartheta_k(t)\|e_{k+2} \neq 0$ almost surely. Moreover, this is actually true almost surely for all $t \in [0,T]$ for any $T>0$.

\begin{lm}
    For any $T,\epsilon>0$, as long as $n>k+3$,
    \begin{equation}
        \pr\left(x_k \in C^{1/2-\epsilon}([0,T]), \vartheta_k(t)-\|\vartheta_k(t)\|e_{k+2} \neq 0 \text{ for all } t \in [0,T]\right)=1.
    \end{equation}
\end{lm}

\begin{proof}
    The proof is by induction. For the $k=0$ case $x_0(t)$ is a vector-valued Ornstein--Uhlenbeck process, so $x_0 \in C^{1/2-\epsilon}([0,T])$. Additionally, $\|x_0(t)-\|x_0(t)\|e_{2}\| \geq \|(x_0^{(2)}(t),...,x_0^{(n+1)}(t))\|$ which is the norm of the vector with all but the first component. But since $n>3$, by classical facts about Brownian motion (e.g., \cite{gall2016brownian}), this norm is a.s. non-zero for all $t \in [0,T]$.
    
    For the inductive step, $M(t)\in C^{1/2-\epsilon}([0,T])$, and since there was a.s. no division by zero in the algebraic operations producing $x_k(t)$, we clearly have $x_k(t) \in C^{1/2-\epsilon}([0,T])$. Now it suffices to prove that $\tilde{x}_k(t)=(\tilde{x}_k^{(1)}(t),...,\tilde{x}_k^{(n-k)}(t)):=(x_k^{(k+2)}(t),...,x_k^{(n+1)}(t)) \neq 0$ for all $t \in [0,T]$ almost surely. To see this, divide the interval $[0,T]$ into $N+1$ equally spaced points $\{t_i\}_{i=1}^{N}$ with spacing $|t_{i+1}-t_i|=\delta=T/N$. Then using the fact that $\left\{\tilde{x}_k^{(j)}(t_1)\right\}_{j=1}^{n-k}$ are i.i.d. standard Gaussians (i.e. Remark \ref{rmk:DE02-reminder}), after fixing any constant $C>0$,
    \begin{equation}
        \begin{split}
            ~&\pr(\tilde{x}_k(t) = 0 \text{ for some } t \in [0,T])\\
            &\leq \pr\left(\|\tilde{x}_k(t)\|_{C^{1/2-\epsilon}} >C \text{ or } \|\tilde{x}_k(t_i)\| < C\delta^{1/2-\epsilon} \text{ for some } i\right)\\
            &\leq \pr(\|\tilde{x}_k\|_{C^{1/2-\epsilon}} >C) +N \pr(\|\tilde{x}_k(t_1)\| < C\delta^{1/2-\epsilon})\\
            &\leq \pr(\|\tilde{x}_k\|_{C^{1/2-\epsilon}} >C) +N \pr(|\tilde{x}_k^{(j)}(t_1)| < C\delta^{1/2-\epsilon}\text{ for } j=1,\dots,n-k)\\
            &\leq \pr(\|\tilde{x}_k(t)\|_{C^{1/2-\epsilon}} >C) +N (C\delta^{1/2-\epsilon})^{n-k-1}.
        \end{split}
    \end{equation}
    Taking $N \to \infty$ ($\delta = T/N \to 0$), and then $C \to \infty$, as long as $(1/2-\epsilon)(n-k-1)>1$, we have $\pr(\Tilde{x}_k(t) = 0 \text{ for some } t \in [0,T])=0$. Since $n-k>3$, as long as $\epsilon < 1/6$, this argument holds. 
\end{proof}

\begin{rmk}
    The condition $n>k+3$ is really necessary in the previous lemma, and the bottom few entries of the tridiagonal will be discontinuous in small examples. This is related to the classical fact that Brownian motion a.s. does not return to its starting point in dimension $\geq 3$.
\end{rmk}

After removing a $0$ probability subset, we will always assume from now on that $n>k+3$, and $\vartheta_k(t), a_k(t), b_k(t) \in C[0,\infty)$.

\subsection{Chebyshev Polynomials}\label{Chebyshev Polynomial Section}
In this section we collect a few facts and identities about the Chebyshev polynomials of the second kind, which can be found in \cite{abramowitz1948handbook}, that will be used in the proof of Theorem \ref{Main Theorem Statement}. 

The Chebyshev polynomials of the second kind $U_k(x)$ are orthogonal on $L^2([-1,1],\frac{2}{\pi} \sqrt{1-x^2}dx)$. They have the three term recurrence, $U_0(1)=1$,
\begin{equation*}
    U_{k+1}(x) = 2x U_k(x)-U_{k-1}(x).
\end{equation*}
Let us also define $P_k(x) = U_k(x/2)$ which are monic, orthonormal on $L^2([-2,2],\frac{1}{2\pi}\sqrt{1-\frac{x^2}{4}} dx)$, and satisfy the recurrence, $P_0(1)=1$,
\begin{equation}
    P_{k+1}(x) = x P_k(x)-P_{k-1}(x).
    \label{Three term reccurence}
\end{equation}
From the recursion its easy to prove the following formula by induction
\begin{equation}
    P_k(x)=\sum_{k'=0}^{[k/2]} (-1)^{k'}\binom{k-k'}{k'} x^{k-2k'}.
    \label{P_k(x) explicit formula}
\end{equation}
Also, $P_k(x)$ satisfies the following product-sum identity,
\begin{equation*}
        P_{j}(x)P_k(x) = \sum_{\ell \leq [(j+k)/2]} P_{j+k-\ell}(x).
\end{equation*}
As a direct corollary, $P_{k}(x)^2-P_{k-1}(x)^2=P_{2k}(x)$, and $(P_{k}(x)-P_{k-2}(x))P_{k-1}(x)=P_{2k-1}(x)$.

\subsection{Partitions, Non-crossing Partitions, and Free Probability}\label{Non-crossing Partitions}

In this section we recall the definition of non-crossing pair partitions, and a generalization of the celebrated asymptotic freedom result for GOE matrices, for more background see \cite{Nica_Speicher_2006}. We purposely avoid the language of free probability for readability.

\begin{defn}
    A \textit{partition} $\pi = \{V_1,\dots ,V_k\}$ of a set $S$ will denote a set of nonempty disjoint subsets $V_i \subset S$ for $i=1,\dots ,k$ s.t. $\cup_{i=1}^{k} V_i = S$. The $V_i$ are the \textit{blocks} of the partition. The set of partitions of $S$ will be denoted $\mathcal{P}(S)$.
\end{defn}

\begin{defn}
    The set of non-crossing partitions $NC[n]$ is the subset of all partitions $\pi = \{V_1,\dots ,V_k\}$ of $[n]=\{1,\dots ,n\}$ s.t. for all choices $1\leq i_1<i_2<i_3< i_4 \leq n$ we cannot have $i_1,i_3 \in V_j$ and $i_2,i_4 \in V_{j'}$ for $j \neq j'$.
\end{defn}

\begin{defn}
    The set of  pair partitions $\mathcal{P}_2(S)$ is the set of $\pi=\{V_1,\dots ,V_k\} \in \mathcal{P}(S)$ with $|V_i|=2$ for $i=1,\dots ,n$. If $n$ is odd note that $\mathcal{P}_2(S) = \emptyset$.
\end{defn}

\begin{defn}
    The set of non-crossing pair partitions $NC_2[n]$ is the set of $\pi=\{V_1,\dots,V_k\} \in NC[n]$ with $|V_i|=2$ for $i=1,\dots,n$. If $n$ is odd note that $NC_2[n] = \emptyset$
\end{defn}

\begin{thm}\label{Semicircular system asymptotics}
    For each $n$, let $(M_1,\dots,M_k)$ be a jointly Gaussian family of GOE random matrices such that
    \begin{equation}
        \begin{split}
            \E\left[(M_i)_{\ell_1,\ell_2} (M_j)_{\ell_1',\ell_2'}\right]= \left(\delta_{(\ell_1,\ell_2),(\ell_1',\ell_2')}+\delta_{(\ell_1,\ell_2),(\ell_2',\ell_1')}\right) \varphi_{i,j}.
        \end{split}
    \end{equation}
    Then,
    \begin{equation}
        \begin{split}
            \frac{1}{n^{1+\ell/2}}\E\left[\mathrm{Tr}(M_{i_1}\dots M_{i_{\ell}})\right] = \sum_{\pi \in NC_2[\alpha]} \prod_{(j,j') \in \pi} \varphi_{i_j,i_{j'}}+O(n^{-1}),
        \end{split}
    \end{equation}
    as $n \to \infty$.
\end{thm}

\subsection{Concentration for Empirical Distribution Moments}\label{Factorization Section}

Here we recall the following central limit theorem and deduce an easy corollary.
\begin{thm}\label{Moment Central Limit Theorem}
Let $M(t)$ be the stationary GOE processes. Then for $k$ fixed times $t_1,\dots ,t_k$ and exponents $j_1, \dots ,j_k$ the vector 
\begin{equation*}
    \left(n^{-j_1/2}\left(\mathrm{Tr}(M(t_1)^{j_1})-\E[\mathrm{Tr}(M(t_1)^{j_1})]\right),\dots,n^{-j_2/2}\left(\mathrm{Tr}(M(t_k)^{j_k})-\E[\mathrm{Tr}(M(t_k)^{j_k})]\right)\right),
\end{equation*}
converges both weakly and in the sense of moments to a Gaussian vector, as $n\to\infty$. Moreover, as a process
\begin{equation*}
    \left(n^{-j_1/2}\left(\mathrm{Tr}(M(t)^{j_1})-\E[\mathrm{Tr}(M(t)^{j_1})]\right),\dots,n^{-j_k/2}\left(\mathrm{Tr}(M(t)^{j_k})-\E[\mathrm{Tr}(M(t)^{j_k})]\right)\right),
\end{equation*}
converges weakly in $C[0,\infty)$ to a Gaussian process.
\end{thm}

For a proof see for instance \cite[Theorem 4.3.20]{Anderson_Guionnet_Zeitouni_2009}.

\begin{lm}\label{Factorization Lemma}
    Let $M(t)$ be the stationary GOE processes. Then for $k$ fixed times $t_1, \dots, t_k$ and exponents $j_1,\dots,j_k$ we have,
\begin{equation*}
    \E\left[\mathrm{Tr}(M(t_1)^{j_1})\dots\mathrm{Tr}(M(t_k)^{j_k})\right] =  \E\left[\mathrm{Tr}(M(t_1)^{j_1})\right]\dots\E\left[\mathrm{Tr}(M(t_k)^{j_k})\right]+O(n^{(j_1+\dots+j_k)/2+k-1}).
\end{equation*}
\end{lm}

\begin{proof}
    We can write
    \begin{align*}
        &\E\left[\mathrm{Tr}(M(t_1)^{j_1})\dots \mathrm{Tr}(M(t_k)^{j_k})\right] \\
        &= \E\left[\left(\E[\mathrm{Tr}(M(t_1)^{j_1})]+(\mathrm{Tr}(M(t_1)^{j_1})-\E[\mathrm{Tr}(M(t_1)^{j_1})])\right)\right.\\
        &\hspace{4cm}\dots\left.\left(\E[\mathrm{Tr}(M(t_k)^{j_k})]+(\mathrm{Tr}(M(t_k)^{j_k})-\E[\mathrm{Tr}(M(t_k)^{j_k})])\right)\right]\\
        &= \sum_{K \subset [k]} \E\bigg[\prod_{i \in K} \E[\mathrm{Tr}(M(t_i)^{j_i})] \prod_{i' \notin K} (\mathrm{Tr}(M(t_{i'})^{j_{i'}})-\E[\mathrm{Tr}(M(t_{i'})^{j_{i'}})])\bigg]\\
        &= \sum_{K \subset [k]} n^{\sum_{i' \in [k]\backslash K}j_{i'}/2}\prod_{i \in K} \E[\mathrm{Tr}(M(t_i)^{j_i})] \E\bigg[\prod_{i' \notin K} n^{-j_{i'}/2}(\mathrm{Tr}(M(t_{i'})^{j_{i'}})-\E[\mathrm{Tr}(M(t_{i'})^{j_{i'}})])\bigg]\\
        &= \E[\mathrm{Tr}(M(t_1)^{j_1})]...\E[\mathrm{Tr}(M(t_k)^{j_k})]+O(n^{(j_1+...+j_k)/2+k-1}).
    \end{align*}
    In the last line we used Theorem~\ref{Moment Central Limit Theorem} with the fact that $\E[\mathrm{Tr}(M(t_i)^{j_i})] = O(n^{j_i/2+1})$. 
\end{proof}

\subsection{Approximate Tridiagonal Entries}\label{Approximate thetas}
For future reference, we define an approximation for $\vartheta_k(t)$ which will play a critical role in our analysis.

\begin{defn}
    We define $\Tilde{M}(t):=(I-e_1  e_1^T)M(t)(I-e_1  e_1^T)$. In other words, $\Tilde{M}(t)$ is obtained by zeroing out the first row and column from $M(t)$, and $\Tilde{M}(t)$ is simply an $n\times n$ GOE process embedded in $\R^{(n+1) \times (n+1)}$. 
\end{defn}

\begin{defn}
    Let $\theta(t)=\theta_0(t):= \vartheta_0(t)$. More generally for $k \geq 0$, define $\theta_k(t):= P_k(\Tilde{M}(t)/\sqrt{n})\theta(t)$. Observe that $\theta(t)$ has $0$ first component, and when restricted to the rest of the $n$ components, $\theta(t) \sim \mathbf{OU}_{n}(1,\sqrt{2})$ which is independent of the GOE process $\Tilde{M}(t)$
\end{defn}

\begin{defn}\label{Def of approximate Tridiagonal Entries}
    Define $\Tilde{a}_0(t):=a_1(t)=e_1^TM(t)e_1$, and for $k\geq 1$, define 
    \begin{equation*}
        \Tilde{a}_{k+1}(t):=\sqrt{n}\theta(t)^T P_{2k-1}(\Tilde{M}(t)/\sqrt{n})\theta(t).
    \end{equation*}
    For $k \geq 0$ let 
    \begin{equation*}
        \Tilde{b}_k(t)^2-n:=n\theta(t)^T P_{2k}(\Tilde{M}(t)/\sqrt{n})\theta(t).
    \end{equation*}
\end{defn}

\begin{defn}
    We define two spaces spanned by vectors defined above as
    \begin{equation}
        \Theta_k(t):=\mathrm{Span}(\{\theta_j(t)\}_{j\leq k}),
    \end{equation}
    and
    \begin{equation}
        \mathcal{E}_k(t):= \mathrm{Span}(\{M(t)^{j_1}e_{j_2+2}\}_{0 \leq j_1 \leq k, 0 \leq j_2 \leq k-1}).
    \end{equation}
\end{defn}

\section{Main Result}\label{sec:mainresult}

\subsection{Statement of Main Result}
We begin by stating the main result of this work. 
\begin{thm}[Convergence of Tridiagonal Entries]\label{Main Theorem Statement} 
Fix a $k \in \mathbb{N}$ and let $\{a_j\}_{j=1}^n, \{b_j\}_{j=1}^{n-1}$ be the tridiagonal entry (processes) as in \eqref{eq:tridiagonalmatrix}, from the $n\times n$ tridiagonalized GOE process. Define 
\begin{equation*}
E_n(t):=\left(a_1(t),\dots, a_k(t), \frac{1}{\sqrt{n}}(b_1(t)^2-n),\dots ,\frac{1}{\sqrt{n}}(b_k(t)^2-n)\right).
\end{equation*} 
Then, $E_n(t)$ converges weakly in $C[0,\infty)$ to a vector process $E(t)=(A_1(t),\dots,A_k(t),B_1(t),\dots,B_k(t))$ as $n\to\infty$, where
\begin{equation*}
    \frac{1}{\sqrt{2}}E(t) \sim \mathbf{OU}_{2k}(1,3,\dots,2k-1,2,4,\dots,2k).
\end{equation*} 
More explicitly, the components $A_j(t) \in C[0,\infty)$ and $B_j(t) \in C[0,\infty)$ are jointly-independent centered Ornstein--Uhlenbeck processes with covariances given as $\mathrm{Cov}(A_j(t),A_j(s))=2e^{-(2j-1)|t-s|}$ and $\mathrm{Cov}(B_j(t),B_j(s))=2e^{-2j|t-s|}$, for $j=1,\dots,k$.
\end{thm}

\subsection{Outline for proof of Theorem \ref{Main Theorem Statement}}
The proof of Theorem~\ref{Main Theorem Statement} can be divided into two main steps, Theorems~\ref{Entry Approximation Thm} and~\ref{Weak Convergence of Approximate Tridiagonal Entries}. The proofs of these steps can be found in Sections~\ref{Approximation Section} and \ref{Moment Method Section} respectively. The following theorem is the first step, with detailed proof given in Section~\ref{Approximation Section}. We approximate $E_n(t)$ by a random vector using approximated tridiagonal entries in Definition~\ref{Def of approximate Tridiagonal Entries} whose moments are easier to compute. 
\begin{thm}\label{Entry Approximation Thm}
    For some constants $C_{k,T}$ and $n_0(k,T)$, we almost surely have 
    \begin{gather*}
        \sup_{t \in [0,T]}|a_j(t)-\tilde{a}_j(t)| \leq C(k,T)\frac{\log(n)^j}{\sqrt{n}},\\        
        \sup_{t \in [0,T]}\frac{1}{\sqrt{n}}\left|b_j(t)^2-\tilde{b}_j(t)^2\right| \leq C(k,T)\frac{\log(n)^{2j}}{\sqrt{n}},
    \end{gather*}
    for all $n \geq n_0(k,T)$ and $j \leq k$. In particular, letting 
    \begin{equation*}
    \Tilde{E}_n(t):=\left(\Tilde{a}_1(t),\dots,\Tilde{a}_k(t), \frac{1}{\sqrt{n}}(\Tilde{b}_1(t)^2-n),\dots,\frac{1}{\sqrt{n}}(\Tilde{b}_k(t)^2-n)\right),
    \end{equation*}
    we almost surely have
    \begin{equation}
        \lim_{n \to \infty} \sup_{t \in [0,T]} \|E_n(t)-\Tilde{E}_n(t)\|=0.
    \end{equation}
\end{thm}
\begin{rmk}
    These $\Tilde{a}_j$ and $\Tilde{b}_j$ can be expressed in terms of the Chebyshev polynomials of the second kind. This is not so surprising as $n \to \infty$ the tridiagonal model for G$\beta$E has spectral measure given by the semicircle law, which has Chebyshev polynomials of the second kind as its associated orthogonal polynomials, and the entries of infinite Jacobi matrices can be expressed in terms of the orthogonal polynomials with respect to their spectral measure (for further details see \cite{deift2000orthogonal}).
\end{rmk}

The key observation for proving Theorem~\ref{Entry Approximation Thm}, is that the algorithm for obtaining the vectors $\{\vartheta_j(t)\}_{j=0}^{k}$ from the Householder tridiagonalization procedure, is approximately the result of applying the Gram-Schmidt process to the set of vectors $\left\{(\Tilde{M}(t)/\sqrt{n})^j \theta(t)\right\}_{j=0}^{k}$. However, due to the large dimension $n$, concentration inequalities tells us that the inner products concentrate around their expectation and thus for a fixed $t \in [0,T]$,
\begin{equation*}
    \begin{split}
        \left((\Tilde{M}(t)/\sqrt{n})^{j_1} \theta(t)\right)^T\left((\Tilde{M}(t)/\sqrt{n})^{j_2} \theta(t)\right) &\approx \E\left[\left((\Tilde{M}(t)/\sqrt{n})^{j_1} \theta(t)\right)^T\left((\Tilde{M}(t)/\sqrt{n})^{j_2} \theta(t)\right)\right]\\
        & = \frac{1}{n}\E\left[\mathrm{Tr}\left((\Tilde{M}(t)/\sqrt{n})^{j_1}(\Tilde{M}(t)/\sqrt{n})^{j_2}\right)\right]\\
        &\approx \int_{-2}^{2} x^{j_1} \cdot x^{j_2} \frac{1}{2\pi}\sqrt{4-x^2} dx,
    \end{split}
\end{equation*}
where the last step follows by the semicircle law. As a result, we naturally have the approximation $\vartheta_j(t) \approx \theta_j(t) = P_j(\Tilde{M}(t)/\sqrt{n})\theta(t)$ where $\{P_j(x)\}$ are the orthogonal polynomials with respect to the semicircle measure $\frac{1}{2\pi}\sqrt{4-x^2} \mathbbm{1}_{|x|\leq 2}$.

With slightly more careful approach we obtain $\vartheta_k(t) \approx \frac{\theta_k(t)}{\|\theta_{k-1}(t)\|}$. Thus after plugging into the expressions for $a_j(t)$ and $b_j(t)$, after some calculation, we obtain 
\begin{equation*}
    \frac{1}{\sqrt{n}}a_{j+1}(t) \approx (\theta_{j}(t)-\theta_{j-2}(t))^T\theta_j(t),
\end{equation*}
and
\begin{equation*}
    \frac{1}{n}(b_{j+1}(t)^2-n) \approx \frac{\|\theta_j(t)\|^2}{\|\theta_{j-1}(t)\|^2}-1.
\end{equation*}
Finally the exact expressions for $\Tilde{a}_j(t)$ and $\Tilde{b}_j(t)$ come from applying an algebraic identity for Chebyshev polynomials to the above expressions. The precise concentration and chaining bounds, which are needed to make all the approximations uniform, are deferred to the appendix.

The second step, carried out in detail in Section~\ref{Moment Method Section}, is to prove the weak convergence for the  approximate entries of our tridiagonal matrix.

\begin{thm}[Weak Convergence of Approximate Tridiagonal Entries]\label{Weak Convergence of Approximate Tridiagonal Entries} 
Fix a $k \in \mathbb{N}$ and $T>0$ and define a vector of approximated tridiagonal entry processes,
\begin{equation*}
    \Tilde{E}_n(t):=\left(\Tilde{a}_1(t),\dots ,\Tilde{a}_k(t), \frac{1}{\sqrt{n}}(\Tilde{b}_1(t)^2-n),\dots ,\frac{1}{\sqrt{n}}(\Tilde{b}_k(t)^2-n)\right).
\end{equation*} Then, $\Tilde{E}_n(t)$ converges weakly in $C[0,T]$ to a vector process $E(t)=(A_1(t),\dots ,A_k(t),B_1(t),\dots,B_k(t))$, distributed as $\frac{1}{\sqrt{2}}E(t) \sim \mathbf{OU}_{2k}(1,3,\dots ,2k-1,2,4,\dots ,2k)$. More explicitly, the components $A_j(t) \in C[0,T]$ and $B_j(t) \in C[0,T]$ are jointly-independent centered Ornstein--Uhlenbeck processes with covariances $\mathrm{Cov}(A_j(t),A_j(s))=2e^{-(2j-1)|t-s|}$ and $\mathrm{Cov}(B_j(t),B_j(s))=2e^{-2j|t-s|}$, for $j=1,\dots ,k$.
\end{thm}

The proof of Theorem \ref{Weak Convergence of Approximate Tridiagonal Entries} can be broken down into the convergence of finite dimensional marginals, and showing tightness. Tightness will follow from quantitative continuity estimates which are proven in the same section of the appendix as the chaining argument for the first step above. 

The convergence of finite dimensional distributions will follow from the convergence of moments. A key quantity in the calculation will be $\E[\mathrm{Tr}(P_{j_1}(\Tilde{M}(t_1)/\sqrt{n}))P_{j_2}(\Tilde{M}(t_2)/\sqrt{n}))]$. From the semicircle law we know the result is $\approx \delta_{j_1,j_2}$ when $t_1=t_2$. A combinatorial argument involving non-crossing partitions and an explicit formula for $P_j(x)$, allows us to compute the value for any $t_1$ and $t_2$. We will prove that
\begin{equation}\label{Orthogonality Computation}
    \frac{1}{n}\E\left[\mathrm{Tr}\left(P_{j_1}(\Tilde{M}(t_1)/\sqrt{n}))P_{j_2}(\Tilde{M}(t_2)/\sqrt{n})\right)\right] \approx \delta_{j_1,j_2}e^{-j_1 |t_1-t_2|}.
\end{equation}
\noindent
After this preliminary computation, Wick's formula will express joint moments of $\Tilde{E}_n(t)$ as a sum indexed by pairings $\pi \in \mathcal{P}_2(\{1,2\}\times J)$ for some index set $J$. Another combinatorial argument allows us to convert this sum into another one indexed by permutations in $\mathcal{S}(J)$. At this point the leading order terms are extracted through combining \eqref{Orthogonality Computation} together with Lemma~\ref{Factorization Lemma}. \\

Combining our two steps, we can prove the main theorem as follows:

\begin{proof}[Proof of Theorem \ref{Main Theorem Statement} given Theorems \ref{Entry Approximation Thm} and \ref{Weak Convergence of Approximate Tridiagonal Entries}]

For any $f \in C_b(C[0,T])$,
\begin{equation}
    \E[f(E_n)] = \E[f(\Tilde{E}_n)]+\E[f(E_n)-f(\Tilde{E}_n)]
\end{equation}
but $\E[f(\Tilde{E}_n)] \to \E[f(E)]$ by Theorem~\ref{Entry Approximation Thm}, and $\E[f(E_n)-f(\Tilde{E}_n)] \to 0$ by the bounded convergence theorem, since $f$ is bounded and $f(E_n)-f(\Tilde{E}_n) \to 0$ a.s. by continuity of $f$ and the fact that $\|E_n-\Tilde{E}_n\|_{C[0,T]} \to 0$ by Theorem~\ref{Weak Convergence of Approximate Tridiagonal Entries}. Thus, $E_n \to E$ weakly in $C[0,T]$.

Next we can equip $C[0,\infty)$ with the metric $d(f,g):= \sum_{k=1}^{\infty} \frac{1}{2^k} \max_{t \in [0,k]} (1 \wedge |f(t)-g(t)|)$. Moreover take a sequence of functions $\eta_m(t) \in C[0,\infty)$ with $\eta_m \equiv 1$ on $[0,m]$, $\eta \equiv 0$ outside of $[0,2m]$, and $\|\eta\|_{L^{\infty}} =1$. Fixing any $f \in C_b(C[0,\infty))$, for all $\nu(t) \in C[0,2m]$, we can define $f_m(\nu):= f(\eta_m \nu)$. We clearly have $f_m \in C_b(C[0,2m])$. From the triangle inequality we now have,
\begin{equation*}
    \begin{split}
        &\limsup_{n\to \infty}|f(E_n)-f(E)|\\
        &\leq \limsup_{n\to \infty}|f_m(E_n)-f_m(E)|+\limsup_{n\to \infty} \|f\|d(\eta_m E_n,E_n)\\
        & \leq \|f\|2^{1-m}.
    \end{split}
\end{equation*}
Since this is true for all $m$, we have $f(E_n) \to f(E)$.
\end{proof}


\section{Numerical Experiments}\label{Numerics}

We perform Householder tridiagonalization (Algorithm~\ref{alg:householder}) on the $\beta=1,2,4$ G$\beta$E processes to obtain tridiagonal matrix valued processes. We compare various statistics of the tridiagonalized G$\beta$E process, against its limiting process, a tridiagonal matrix valued process based on our tridiagonal model described in Section~\ref{sec:mainresult}. 


\subsection{Diagonal and Off-diagonal Entry Processes }\label{sec:abcompare}

We begin by comparing the entry processes directly. In Theorem~\ref{Main Theorem Statement} we show that the $j^\text{th}$ diagonal entry process of the tridiagonalized GOE process, $a_j(t)$, is weakly converging ($n\to\infty$) to an Ornstein-Uhlenbeck process $A_j(t)$ defined by
\begin{equation}\label{eq:aj}
    d A_j(t) = -(2j-1) A_j(t) dt + 2\sqrt{2j-1}dW(t),
\end{equation}
where $dW(t)$ is the white noise. This is $\sqrt{2}\mathrm{OU}(2j-1)$ following Definition~\ref{Ornstein-Uhlenbeck Process}. In Figure~\ref{fig:meanvara} we first compare the mean and variance of $a_j(t)$ and $A_j(t)$ at each $t$, for some arbitrarily picked $j$. The two plots essentially show that $a_j(t)$ and $A_j(t)$ have nearly the same mean and variance at each time. 

\begin{figure}[h]
    \centering
    \includegraphics[width=0.95\textwidth]{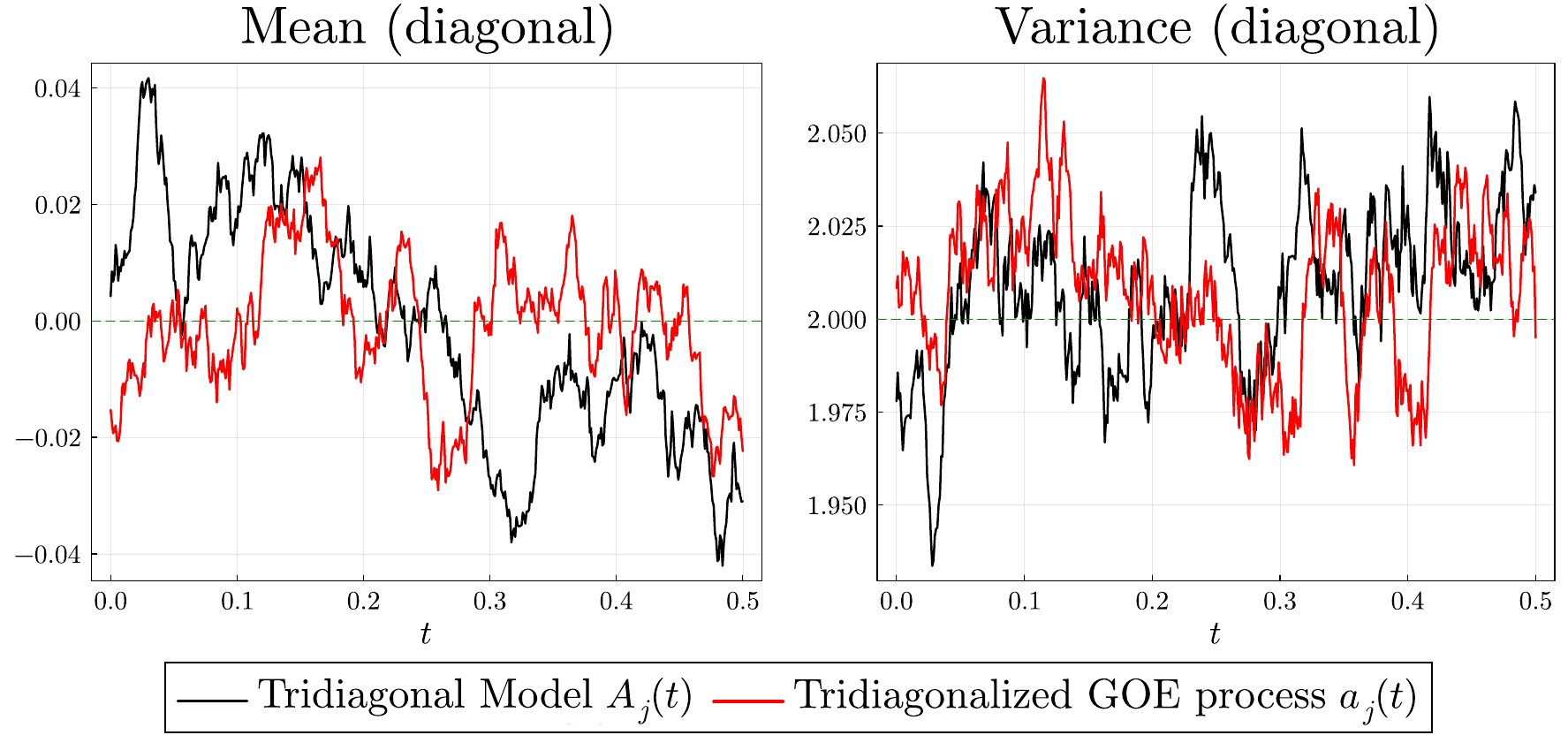}
    \caption{Comparing empirical means and variances of $A_j(t)$ (black) and $a_j(t)$ (red) for an arbitrary choice $j=9$. One can observe that the mean is around zero and the variance is near 2 (green dashed line) for both processes at each time, as their stationary distiribution is $\mathcal{N}(0, 2)$. We used 10000 samples of $2000\times 2000$ tridiagonalized GOE processess and 10000 samples of the Ornstein-Uhlenbeck process \eqref{eq:aj} with $n=2000$. For both samples we sampled from $t=0$ to $t=0.5$ with $dt=10^{-3}$.}\label{fig:meanvara}
\end{figure}

\begin{figure}[h]
    \centering    \includegraphics[width=0.5\textwidth]{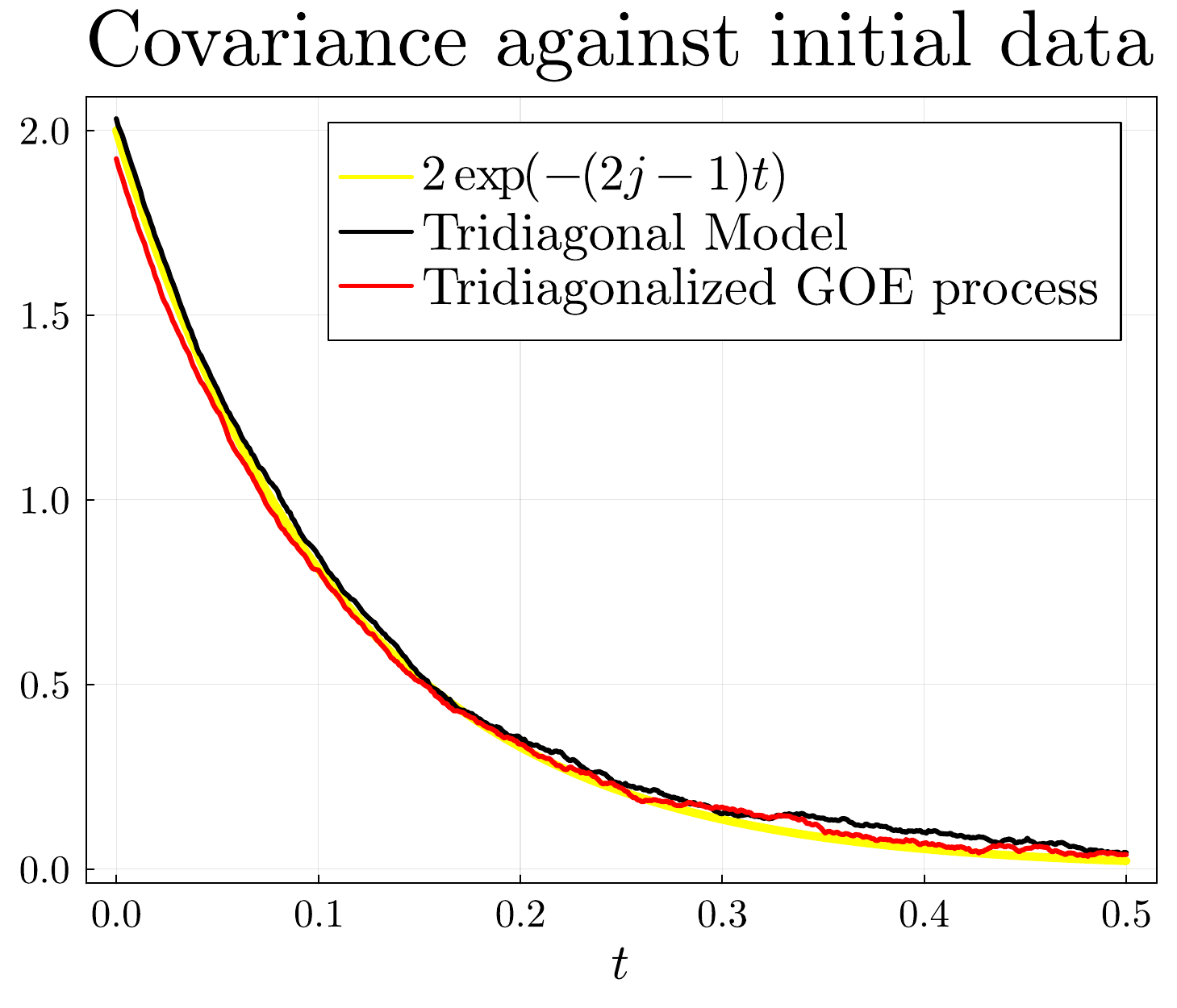}
    \caption{Sample covariances $\text{Cov}(a_j(0), a_j(t))$ and $\text{Cov}(A_j(0), A_j(t))$ with an arbitrary choice of $j=5$. The yellow line is the theoretical covariance of the process $A_j(t)$. We used the same samples used in Figure~\ref{fig:meanvara}.}\label{fig:cov_a}
\end{figure}

We then compare the covariance of the two processes (against their initial data) in Figure~\ref{fig:cov_a}. Theorem~\ref{Main Theorem Statement} suggests that these processes should have covariance $e^{-(2j-1)|t_1-t_2|}$ for given two $t$ values $t_1, t_2$, as $n\to\infty$. 
\clearpage

Similarly, we compare the subdiagonal entry processes. 
Let $b_j(t)$ be the $j^\text{th}$ subdiagonal entry process of the tridiagonalized GOE process. From Theorem~\ref{Main Theorem Statement} we have $\frac{1}{\sqrt{n}}(b_j(t)^2 - n)$ weakly converging (with $n\to\infty$) to the Ornstein-Uhlenbeck process $B_j(t)$ defined by
\begin{equation*}
    d B_j(t) = -2j B_j(t)dt + 2\sqrt{2j} dW(t). 
\end{equation*}
We compute a stochastic process $\hat{b}_j(t):=\sqrt{\sqrt{n}B_j(t)+n}$ to compare directly with $b_j(t)$. Figure~\ref{fig:meanvarcov_b}  shows that the mean, variance and covariance of the two stochastic processes $\hat{b}_j(t)$ and $b_j(t)$ align, which verifies the weak convergence. 

\begin{figure}[hh]
    \centering
    \includegraphics[width=0.9\textwidth]{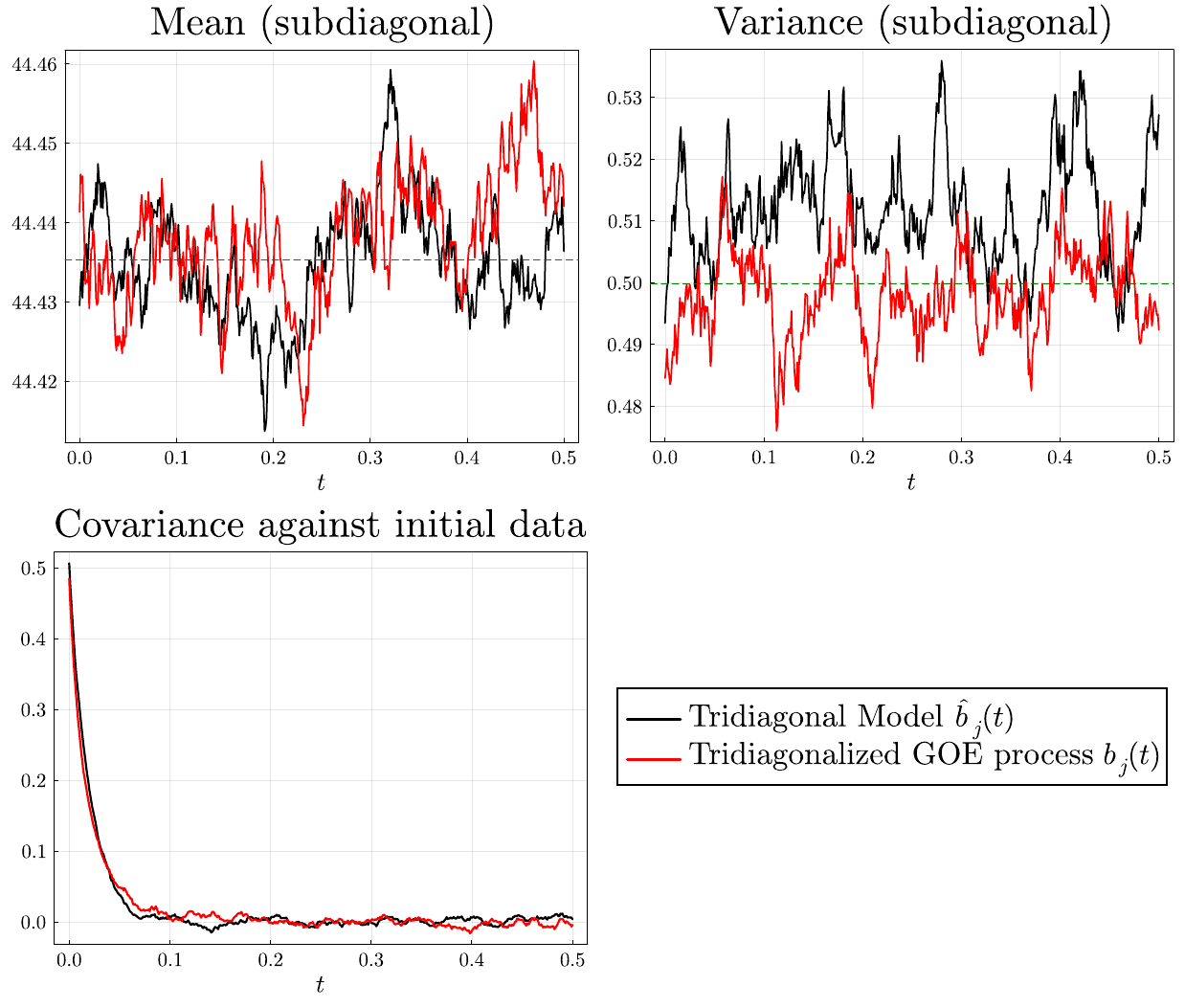}
\caption{Comparing empirical means and variances of $\hat{b}_j(t)$ (black) and $b_j(t)$ (red) for an arbitrary choice $j=25$. One can observe that for each $t$, the mean is around $\sqrt{2}\frac{\Gamma((n-j+1)/2)}{\Gamma((n-j)/2)} \sim \sqrt{n-j} = 44.44$, and the variance is around $n-j-2(\frac{\Gamma((n-j+1)/2)}{\Gamma((n-j)/2)})^2 \sim \frac{1}{2}$, which is represented by the green dashed lines. We used the samples from Figure~\ref{fig:meanvara} for the tridiagonalized GOE model $b_j(t)$. We used 10000 samples of the Ornstein-Uhlenbeck process $B_j(t)\sim \mathrm{OU}(2j)$ and computed $\hat{b}_j(t)$ for each sample.}\label{fig:meanvarcov_b}
\end{figure}

\clearpage 

\subsection{The Diagonal Entry  is not a Gaussian Process}\label{sec:notgaussian}
We show that the diagonal entry process $a_j(t)$ of the tridiagonalized GOE process is not a Gaussian process, when $n$ is not large enough. From \eqref{StationaryTridiagonal} the stationary distribution of $a_j(t)$ should be a Gaussian distribution with mean zero and variance 2. In Figure~\ref{fig:smallkurtosis} we show that $a_j(t)$ is not a Gaussian process by demonstrating nonzero (excess) kurtosis of $a_j(0)+a_j(t)$, despite each $a_j(0), a_j(t)$ being Gaussians.

\begin{figure}[h]
    \centering
    {\includegraphics[width=0.9\textwidth]{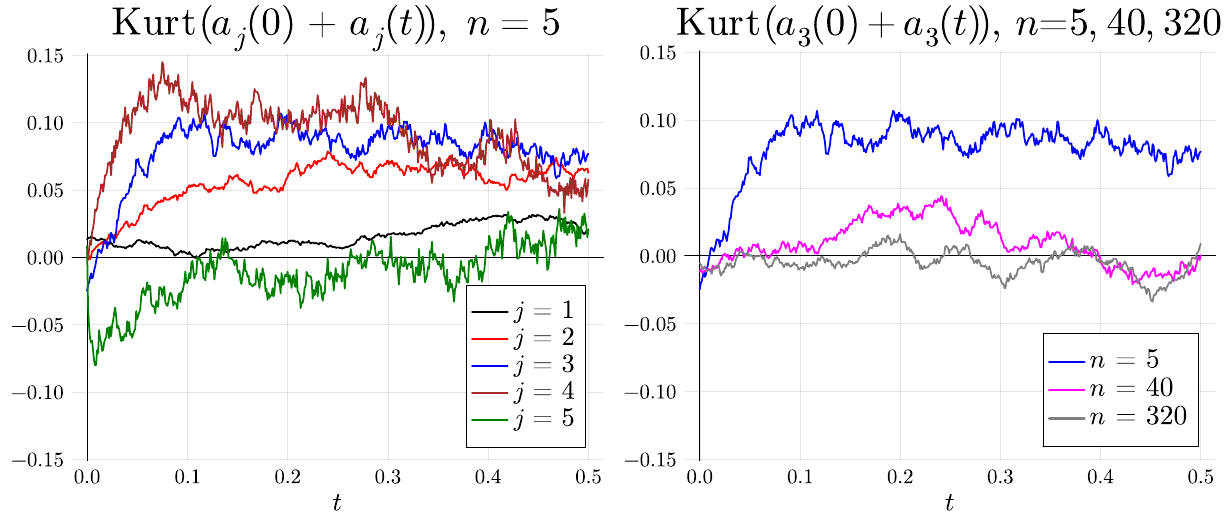}}
    \caption{
    The first plot shows that the sum $a_j(0)+a_j(t)$ is not a Gaussian (for $j>1$) as we observe nonzero (excess) kurtosis at each time. Note that for $j=1$ the process is indeed Gaussian. The second plot illustrates that the process $a_j(t)$ gets closer to a Gaussian process as $n$ gets larger, with an arbitrary choice of $j=3$. (Notice that the blue line represents $a_3(t)$ for $n=5$ in both bottom plots.) Recall from Theorem~\ref{Main Theorem Statement} $a_j(t)$ converges weakly to a Ornstein-Uhlenbeck process as $n\to\infty$. We sampled 100000 tridiagonalized GOE processes for each $n=5, 40, 320$.}\label{fig:smallkurtosis}
\end{figure}

\subsection{Largest Eigenvalue Process}

In this experiment we compare the largest eigenvalue of the G$\beta$E process and the largest eigenvalue process of our tridiagonal model. The largest eigenvalue $\lambda_{\max}(t)$ of the GUE process is a widely studied topic, since it is known to converge as $n\to\infty$ to the Airy$_\text{2}$ process $\mathcal{A}(t)$ under the following renormalization and rescaling,
\begin{equation*}
    \lim_{n\to\infty} n^{1/6}\left(\lambda_{\max}(n^{-1/3}t) - 2\sqrt{n}\right) = \mathcal{A}(t). 
\end{equation*}
(However, intriguingly, the largest eigenvalue process of the GOE process does not align with the Airy$_\text{1}$ process \cite{BoFe08}.)

Computing the largest eigenvalue of an $n\times n$ matrix with large $n$ is computationally expensive as the complexity of eigenvalue algorithms are $O(n^3)$ in general. However we can efficiently approximate the largest eigenvalue of a large GOE matrix by taking the largest eigenvalue of the $10 n^{1/3}\times 10 n^{1/3}$ upper left corner of the tridiagonal model~\eqref{StationaryTridiagonal} \cite[Section 10]{edelman2005random}. In the following experiment, we do the same for each $t$ to efficiently approximate the true largest eigenvalue process. 

\par We sample the largest eigenvalue process $\lambda_{\max{}}(t)$ of (1) the full $n\times n$ GOE process, (2) $k\times k$ upper left corner process of our tridiagonal model (using $A_j(t), B_j(t)$ in Theorem~\ref{Main Theorem Statement}), and (3) $k\times k$ upper left corner of the tridiagonalized GOE process. We compare covariances of the three largest eigenvalue processes against their initial data in Figure~\ref{fig:eig_cov}.

\begin{figure}[hh]
    \centering
    \includegraphics[width=0.9\textwidth]{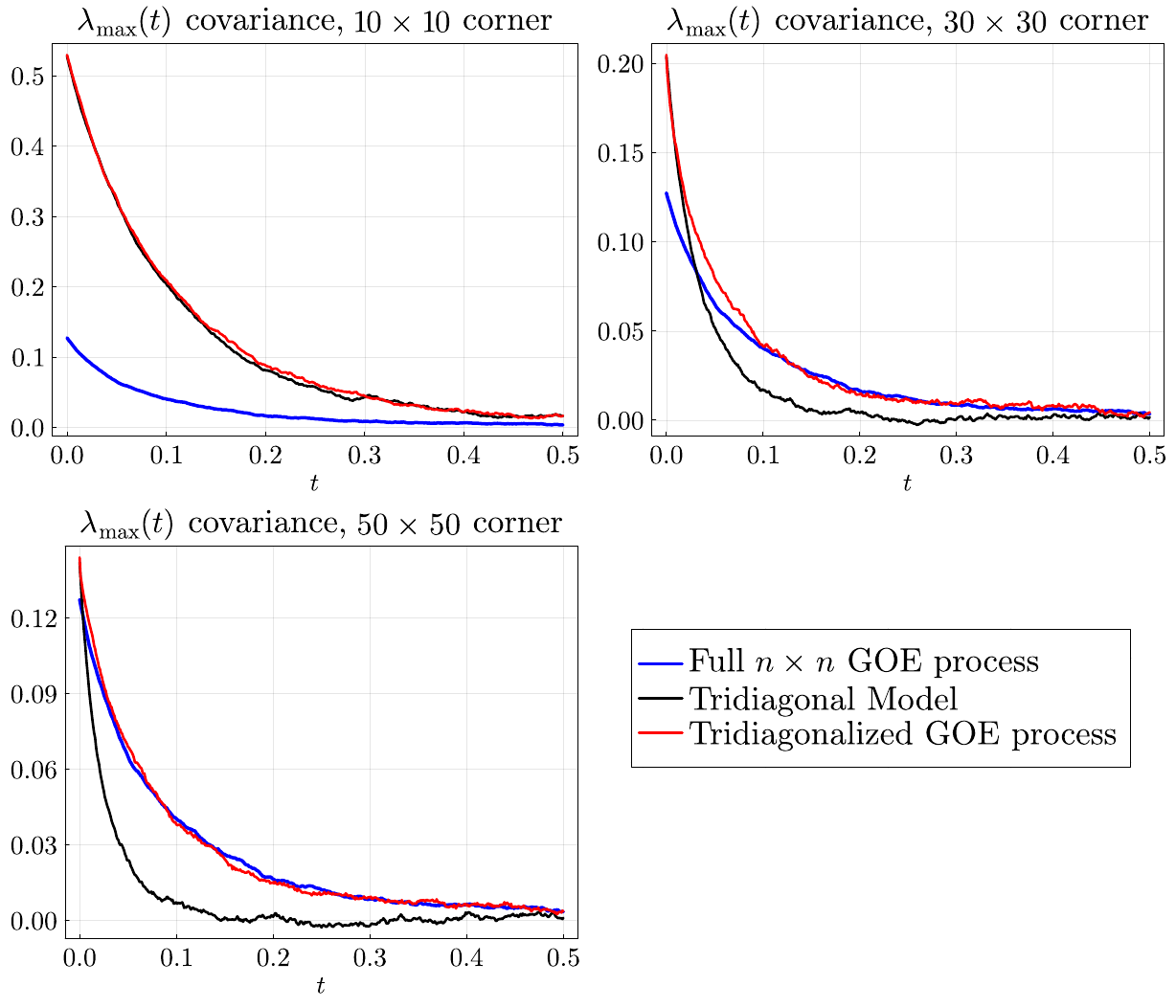}
    \caption{Empirical covariances (against initial data) of the largest eigenvalue processes $\lambda_{\max{}}(t)$ for three different stochastic processes. The blue lines are computed from $\lambda_{\max{}}(t)$ of the full $2000\times 2000$ GOE process. The black lines are covariances of $\lambda_{\max{}}(t)$ of the $k\times k$ tridiagonal model (Theorem~\ref{Main Theorem Statement}) with $n=2000$. The red lines are from $\lambda_{\max{}}(t)$ of the $k\times k$ upper-left corner of the tridiagonalized GOE process. We plot for $k=10, 30, 50$. Observe that for small $k$ the $\lambda_{\max{}}(t)$ of tridiagonal model (black) is close to $\lambda_{\max{}}(t)$ of the tridiagonalized GOE process (red). This shows that the underlying dynamics of the tridiagonalized GOE process is well approximated by our tridiagonal model with $A_j(t), B_j(t)$, when $k$ is small. As $k$ gets larger, the $\lambda_{\max{}}(t)$ of tridiagonalized GOE process starts to approximate the true $\lambda_{\max{}}(t)$ process better, and $\lambda_{\max{}}(t)$ of the tridiagonal model starts to deviate from the other two. For each plot we used 2000 samples of all three processes.}\label{fig:eig_cov}
\end{figure}

\clearpage

\subsection{The cases $\beta=2,4$}

In this section we perform numerical experiment on (tridiagonalized) G$\beta$E processes with $\beta=2,4$. Any numerical simulations with the GOE process can be simply modified to the GUE and GSE processes. We plot covariance against initial data for diagonal entry ($\beta=2$) and off-diagonal entry ($\beta=4$) to confirm that Theorem~\ref{Main Theorem Statement} holds for $\beta=2, 4$. 

\begin{figure}[h]
    \centering
    \includegraphics[width=0.9\textwidth]{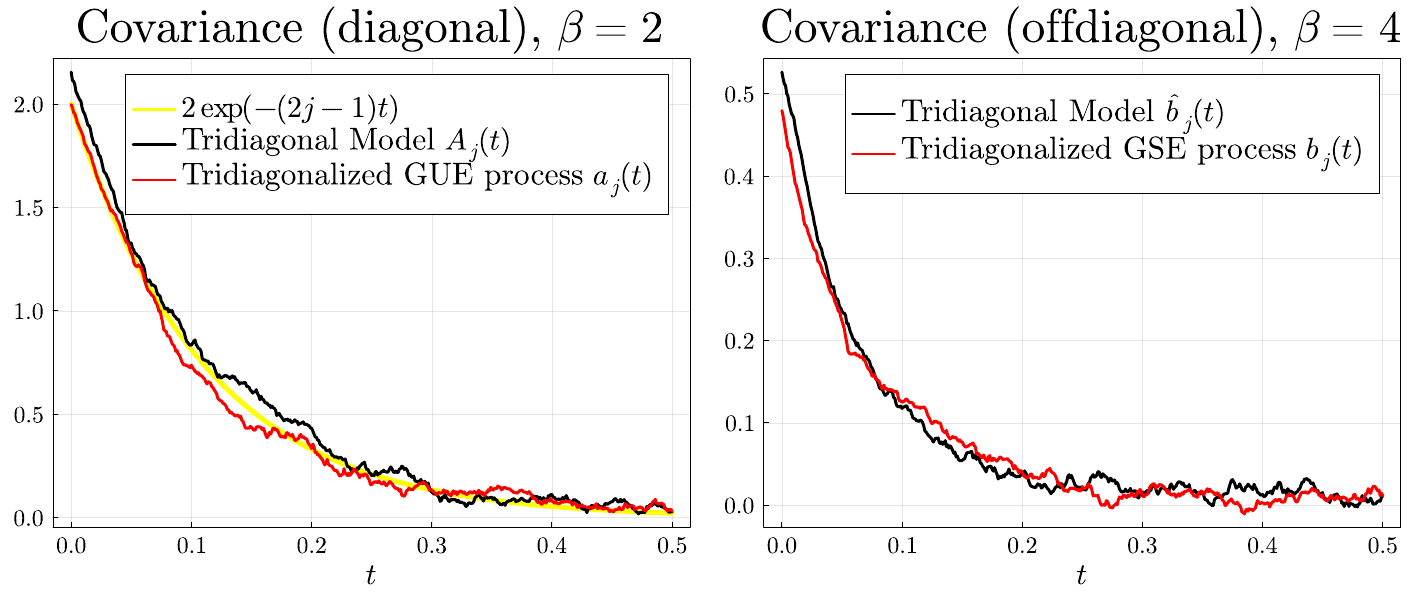}
    \caption{Plots of covariance against $t=0$ of tridiagonalized GUE diagonal process (left) and tridiagonalized GSE off-diagonal process (right). We sampled 1000 tridiagonalized $200\times 200$ GUE and GSE processes, from $t=0$ to $t=0.5$ with $dt=10^{-3}$. For $\beta=2$ we plot covariances of the $j^\text{th}$ diagonal process from the tridiagonalized GUE process (red) and Ornstein-Uhlenbeck process $A_j(t)$ (black), with an arbitrary choice of $j=5$. The yellow line represents analytic covariance for the Ornstein-Uhlenbeck process. For the GSE process we plot covariances of $b_j(t)$ and $\hat{b}_j(t)$ (see Section~\ref{sec:abcompare}) with an arbitrary choice of $j=7$.}\label{fig:eigmeanvar}
\end{figure}

\subsection{Implementation Details}

To sample the entries of the tridiagonalized GOE (GUE, GSE) process we discretized time $t\in[0, 0.5]$ with $dt=0.001$. For each entry we sampled Ornstein-Uhlenbeck process with discretized time and for each $t$ series of Householder reflectors are applied for the tridiagonalization. We used $n$ as large as $2000$ and the entry processes are collected using Xeon-P8 CPU with 16 cores from MIT supercloud server \cite{reuther2018interactive}.

\section{Past Work, Future Directions, and Open Problems}\label{open problem section}

\subsection{Past Work}

First, we briefly summarize two related previous works \cite{AlGu13, HoPa17} on matrix models for $\beta$-DBM. In particular \cite{HoPa17} provides tridiagonal models for $\beta$-DBM in a different direction. To explain their results we first recall that in \cite{DuEd02}, the authors establish a bijection between symmetric tridiagonal matrices and their eigenvalues together with the spectral weights (first entry of each eigenvector). Thus, to construct a symmetric tridiagonal model for $\beta$-DBM, we have the freedom of choosing the evolution of the spectral weights. In \cite{HoPa17}, the authors classify the tridiagonal models for $\beta$-DBM through an SDE involving the spectral weights, for a large class of spectral weights. Moreover, in the frozen (constant) spectral weight case, they provide asymptotics for finite upper left blocks as $n \to \infty$ through a limiting SDE they satisfy. Our result concerns a different situation since the spectral weights for the GOE ensemble are a.s. not constant. Moreover our approach also largely differs in proof techniques. In another paper \cite{li2010beta}, the author proposes a dynamical tridiagonal $\beta$-Hermite ensemble, however, this model does not exactly have eigenvalues evolving according to $\beta$-DBM, and when contrasted with our main result, does not have a form which results from applying the Householder tridiagonalization procedure to the G$\beta$E process.

Lastly in a series of recent works \cite{gorin2024airy,dunkl2024eigenvalues,huang2024convergence}, the authors found explicit formulas for the moments and Laplace transform of the $\beta$-Airy line ensemble by finding formulas for the G$\beta$E corners process and taking appropriate scaling limits. The proofs of these impressive results do not use the tridiagonalization of the G$\beta$E process studied in this paper.

\subsection{Future Directions and Open Problems}

Our proof of Theorem~\ref{Main Theorem Statement} should generalize in a straightforward manner to the $\beta=2,4$ cases. The precise statement for this generalization is the following. \\

\noindent\textbf{Problem 1:} Suppose $M^{(n)}(t)$ is a sequence of stationary $n \times n$ matrix processes evolving according to the G$\beta$E process   for $\beta \in \{2,4\}$. After applying the Householder tridiagonalization to obtain diagonal entries $a_1^{(n)}(t), a_2^{(n)}(t),\dots,a_n^{(n)}(t)$ and off-diagonal entries $b_1^{(n)}(t), b_2^{(n)}(t),\dots,b_n^{(n)}(t)$, prove that for any fixed $k$,
\begin{equation*}
    \begin{split}
        ~&\left(a_1(t),\dots,a_k(t), \frac{1}{\beta \sqrt{n}}(b_1(t)^2-\beta n),\dots,\frac{1}{\beta\sqrt{n}}(b_k(t)^2-\beta n)\right) \\
        &\hspace{5cm}\to \left(A_1(t),\dots,A_k(t),B_1(t),\dots,B_k(t)\right),
    \end{split}
\end{equation*}
as $n\to\infty$, with the entries of the limiting vector being jointly-independent as in the $\beta=1$ case.\\

A less straightforward direction is to try to extend our results to the $\beta$-Laguerre and $\beta$-Jacobi ensembles. Here we briefly pose a problem for the L$\beta$E process. \\

\noindent\textbf{Problem 2:} Take $\beta \in \{1,2,4\}$. Suppose $M^{(n)}(t)$ is a sequence of $n \times n$ stationary matrix process evolving according to $\beta$-Laguerre diffusion. Can we extract a limit of the $k \times k$ upper left corner of $M^{(n)}(t)$ for $k$ fixed and $n \to \infty$?   \\

The next problem is to prove that the pattern extends to entries further down the tridiagonal. One way to formulate this problem would be the following.\\

\noindent \textbf{Problem 3:}
For what threshold $k(n) \to \infty$ as $n \to \infty$, can we prove that for any $k_1(n),\dots,k_j(n) \to \infty$ and $k_1(n),\dots,k_j(n) \leq k(n)$,
\begin{equation*}
    \begin{split}
        ~&\left(a_{k_1(n)}^{(n)}(t/k_1(n)),\dots,a_{k_j(n)}^{(n)}(t/k_j(n)), \frac{1}{\sqrt{\beta(n-k_1(n))}} \left(b_{k_1(n)}^{(n)}(t/k_1(n))^2-\beta (n-k_1(n))\right),\right. \\
        &\hspace{5cm}\dots, \left.\frac{1}{\sqrt{\beta(n-k_1(n))}}\left(b_{k_j(n)}^{(n)}(t/k_j(n))^2-\beta (n-k_j(n))\right)\right)\\
        &\hspace{8cm}\to (A_1(t),\dots,A_j(t),B_1(t),\dots,B_j(t))
    \end{split}
\end{equation*}
weakly in $C[0,\infty)$, where the components $A_i(t) \in C[0,\infty)$ and $B_i(t) \in C[0,\infty)$ are independent centered Ornstein--Uhlenbeck process with covariances given as $\mathrm{Cov}(A_i(t),A_i(s))=2e^{-2|t-s|}$ and $\mathrm{Cov}(B_j(t),B_j(s))=2e^{-2|t-s|}$ for $i=0,1,\dots,j$. \\

The proofs in this paper should be fairly easily modified if one takes $k(n)=(\log(n))^{\alpha}$ for some $\alpha$ for instance, however we would like to be able to take $k(n)>>n^{1/3}$ to extract spectral information. However, our numerical tests do not provide strong evidence for this regime. It is still not entirely clear whether this is due to numerical instability, since for large $n$ the tridiagonal entries have large oscillations as $k$ gets larger which causes inaccurate simulation.

Finally we state the problem which initially motivated this paper.\\

\noindent \textbf{Problem 4:} If $\lambda_1^{(n)}(t)>\dots>\lambda_k^{(n)}(t)$ are the $k$ largest eigenvalues of $\beta$-Dyson Brownian motion on $n$ particles as in \eqref{Dyson BM}, do
\begin{equation*}
    n^{1/6}(2\sqrt{n}-\lambda_j^{(n)}(t/(2n^{1/3}))) \to \lambda_j(t)
\end{equation*}
as $n \to \infty$, for $j=1,\dots,k$ jointly in distribution, where $\lambda_1(t)\leq \lambda_2(t) \leq \dots$ are the eigenvalues of a time dependent stochastic Airy operator appearing in \cite{ramirez2011beta}. To the knowledge of the authors, there is no consensus even on a conjecture for what the limiting dynamic stochastic operator should be, or even if it exists.\\

We end the section by discussing a natural candidate for the dynamical stochastic Airy operator suggested by Theorem \ref{Main Theorem Statement}, followed by a remark explaining why it is not the operator required in Problem 4.

If one assumes that the pattern of entries in Theorem \ref{Main Theorem Statement} continues down the entire tridiagonal matrix, than following the informal derivation from tridiagonal matrix to stochastic differential operator proposed in \cite{EdSu07}, then one is led to the dynamical Airy stochastic operator of definition \ref{def:stochastic-airy-op}. 

\begin{defn}\label{def:stochastic-airy-op}
    We define \textit{Ornstein--Uhlenbeck space-time white noise} as a one parameter family of random continuous functions of time into space distributions $\xi_{a}:[0,\infty)\to \mathcal{D}'((0,\infty))$, such that for any test functions $f_1,\dots,f_k \in C_c^{\infty}((0,\infty))$, $(\la \xi_{a}(t_1),f_1 \ra,\dots,\la \xi_{a}(t_k), f_k \ra)$ is a centered Gaussian process, and
    \begin{equation}
        \E[\la \xi_{a}(t), f_i\ra \la \xi_{a}(s), f_j\ra] = \int_{0}^{\infty} e^{-a|t-s|x} f_i(x)f_j(x) dx.
    \end{equation}
\end{defn}

\begin{rmk}
    One can verify that $\xi_{a} \in C([0,\infty):C^{-1/2-\epsilon}(0,\infty))$ for any $\epsilon>0$, by a Kolmogorov continuity argument. One can also naturally define $\la \xi_a(t), f \ra$ as a well-defined Gaussian process for any $f \in L^2(0,\infty)$.
\end{rmk}

\begin{defn}
    We define a time--dependent $\beta$-stochastic operator formally expressed as
    \begin{equation}
        \mathcal{H}_{\beta}(t) := -\frac{d^2}{dx^2}+x+\frac{2}{\sqrt{\beta}}\xi_1(t).
    \end{equation}
    More precisely, for $f \in H^{1,1}_{0}(0,\infty):= \{u \in L^2(0,\infty): u',xu \in L^2(0,\infty), u(0)=0\}$, define $\mathcal{H}_{\beta}(t)f$ as the distribution such that for any $g \in C_0^{\infty}(0,\infty)$,
    \begin{equation}
        \la \mathcal{H}_{\beta}(t) f,  g \ra := \int_{0}^{\infty} (f'(x) g'(x)+xf(x)g(x))dx+\frac{2}{\sqrt{\beta}}\int_{0}^{\infty} f(x)g(x) \xi_1(t)[dx],
    \end{equation}
    with $\int_{0}^{\infty} f(x)g(x) \xi_1(t)[dx] = \la \xi_1(t), f(x)g(x) \ra$.
\end{defn}
The following two lemmas will be used to define the eigenvalues of $\mathcal{H}_{\beta}(t)$. The proofs can be found in \cite{ramirez2011beta}.
\begin{lm}
    For any $f, g\in H_0^{1,1}(0,\infty)$, the quadratic form,
    \begin{equation}
        \la f, \mathcal{H}_{\beta}(t) g\ra,
    \end{equation}
    is well defined as $\lim_{n \to \infty } \la f_n, \mathcal{H}_{\beta}(t) g\ra$ for any $f_n \in C_c^{\infty}(0,\infty)$ with $f_n \to f.$
\end{lm}

\begin{lm}
    For any fixed $t$, the infimum,
    \begin{equation}
        \lambda_1(t):=\inf_{f \in H_0^{1,1}, ||f||_{L^2}=1} \la f, \mathcal{H}_{\beta}(t) g\ra,
    \end{equation}
    is attained by some $f_{1,t}(x) \in H_0^{1,1}$. We call the $\lambda_0(t)$ the smallest eigenvalue of $\mathcal{H}_{\beta}(t)$.
\end{lm}

The smallest $k$ eigenvalues of $\mathcal{H}_{\beta}(t)$ can now be defined inductively through the usual variational formulation. 

\begin{figure}[hh]
    \centering
    \includegraphics[width=0.6\textwidth]{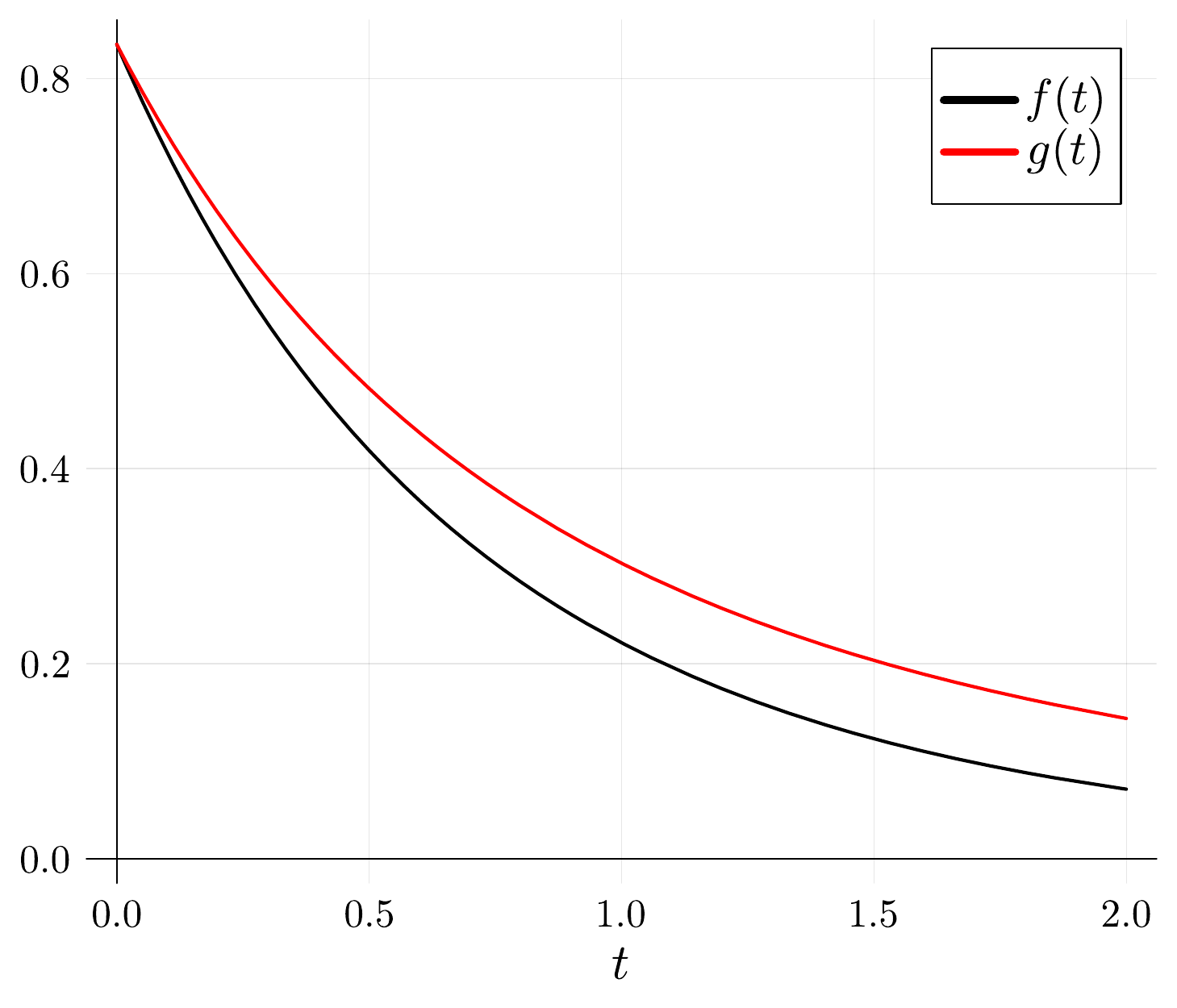}
    \caption{In this figure we  compare $f(t)=2\int_0^{\infty} \psi_1(x)^4 e^{-xt} dx$ with $g(t)=\int_0^{\infty}   \frac{e^{- xt}}{x} \psi_1(x)^2dx$ on the interval $t \in [0,2]$. $f(t)$ and $g(t)$ are the two candidate first order perturbations of $\mathrm{Cov}(\lambda_1(0),\lambda_1(t))$. As expected they agree at $t=0$, but clearly differ for $t>0$.}
\end{figure}\label{fig:Airy-Perturb-Graph}
\begin{rmk}[$\mathcal{H}_{\beta}(t)$ is not the correct dynamical stochastic Airy operator: a perturbative calculation]
    In this remark we discuss a perturbative calculation showing that the operator of definition \ref{def:stochastic-airy-op} should not be the conjectured limit of Problem 4 \cite{pierreprivatecomm}. 

    A first order perturbation of the operator of Definition \ref{def:stochastic-airy-op} about the deterministic Airy operator, $-\frac{d^2}{dx^2}+x$ suggests the expansion,
    \begin{equation*}
    \lambda_i(t) = -a_i + \frac{2}{\sqrt{\beta}} \langle \xi_1(t),\psi_i^2 \rangle + \dots,
    \end{equation*}
    where $\psi_i(x):=\frac{\mathrm{Ai}(x+a_i)}{\mathrm{Ai'}(a_i)}$ is the $i$th eigenfunction of the Airy operator (Recall $\mathrm{Ai}(x)$ is the Airy function and $a_i$ is the $i$th zero of the Airy function). So plugging in this perturbation appropriate for large $\beta$, we have
    \begin{equation}\label{eq:perturb-cov-1}
        \mathrm{Cov}(\lambda_i(t_1),\lambda_j(t_2)) \approx \frac{4}{\beta} \int_0^{\infty} \psi_i(x)^2\psi_j(x)^2 e^{-x|t_1-t_2|} dx.
    \end{equation}
    On the other hand, there already exist rigorously derived expressions for the first order in $\beta^{-1/2}$ perturbation of the two time covariance between eigenvalues in the Airy $\beta$--line ensemble (i.e. \cite[Theorem 1.2]{GorinPerturbation} and \cite[Equation (126)]{PierrePerturbations})  ,
    \begin{equation}\label{eq:perturb-cov-2}
         \mathrm{Cov}(\lambda_i(t_1),\lambda_j(t_2)) = \frac{2}{\beta}\int_0^{\infty}   \frac{e^{- x|t_1-t_2|}}{x} \psi_i(x) \psi_j(x)dx+o(\beta^{-1}).
    \end{equation}
    We now give an argument why \eqref{eq:perturb-cov-1} and \eqref{eq:perturb-cov-2} do not agree in general. Take $i=j=1$, $t_1=0$, and $t_2=t \gg 1$. Observe that $\psi_1(0)=\frac{\mathrm{Ai}(a_1)}{\mathrm{Ai}'(a_1)}=0$, and $\psi_1'(0)=\frac{\mathrm{Ai}'(a_1)}{\mathrm{Ai}'(a_1)}=1$, so $\psi_1(x) =x+O(x^2)$ around $x=0$. Moreover, in the $t \gg 1$ regime the integral in \eqref{eq:perturb-cov-1} is dominated by the behavior around $x=0$. Thus,
    \begin{equation*}
        \int_0^{\infty} \psi_1(x)^4 e^{-tx}dx \sim \int_0^{\infty} x^4 e^{-tx}dx \sim \frac{4!}{t^5},
    \end{equation*}
    while by similar reasoning,
    \begin{equation*}
        \int_0^{\infty} \frac{\psi_1(x)^2}{x} e^{-tx}dx \sim \int_0^{\infty} x e^{-tx}dx \sim \frac{1!}{t},
    \end{equation*}
    thus clearly the two expressions \eqref{eq:perturb-cov-1} and \eqref{eq:perturb-cov-2} are not equal. We also provide a graph obtained numerically to verify the discrepancy even at small times (see Figure \ref{fig:Airy-Perturb-Graph}).
\end{rmk}

\section{Proof of Theorem \ref{Entry Approximation Thm}}\label{Approximation Section}

The main ingredient for proving Theorem \ref{Entry Approximation Thm} is the following proposition, which is proved at the end of this section.  

\begin{prop}\label{Quantitative Approximation Probability }
For any fixed $k$, there exist constants $C(T,k), C'(T,k)$ such that
    \begin{equation*}
        \pr\bigg(\sup_{t \in [0,T]}|a_j^{(n)}(t)-\tilde{a}_j^{(n)}(t)|>C(T,k)\frac{\log(n)^j}{\sqrt{n}} \text{ for some } j\leq k\bigg) < C'(T,k)n^{-2},
    \end{equation*}
    and
    \begin{equation*}
        \pr\bigg(\sup_{t \in [0,T]} \frac{1}{\sqrt{n}}|b_j^{(n)}(t)^2-\tilde{b}_j^{(n)}(t)^2|>C(T,k)\frac{\log(n)^{2j}}{\sqrt{n}} \text{ for some } j \leq k\bigg) < C'(T,k)n^{-2}.
    \end{equation*}
\end{prop}

\begin{proof}[Proof of Theorem \ref{Entry Approximation Thm}, given Proposition \ref{Quantitative Approximation Probability }] 
We have 
\begin{equation*}
    \sum_{n=1}^{\infty}\pr\bigg(\sup_{t \in [0,T]}|a_j^{(n)}(t)-\tilde{a}_j^{(n)}(t)|>C(T,k)\frac{\log(n)^j}{\sqrt{n}} \text{ for some } j\leq k \bigg) \leq \sum_{n=1}^{\infty} C'(T,k) n^{-2} <\infty.
\end{equation*}
Thus, by the Borel-Cantelli lemma we have
\begin{equation*}
    \pr\bigg(\sup_{t \in [0,T]}|a_j^{(n)}(t)-\Tilde{a}_j^{(n)}(t)|>C(T,k)\frac{\log(n)^j}{\sqrt{n}} \text{ for some } j\leq k, \text{ for infinitely many }n \bigg) = 0,
\end{equation*}
and in particular almost surely there exists an $n_0(T,k)$ such that $\sup_{t \in [0,T]} |a_j^{(n)}(t)-\tilde{a}_j^{(n)}(t)|\leq C(T,k)\frac{\log(n)^j}{\sqrt{n}}$ for all $n \geq n_0(T,k)$. An identical argument proves there exists an $n_0'(T,k)$ such that $\sup_{t \in [0,T]}|b_j^{(n)}(t)^2-\tilde{b}_j^{(n)}(t)^2|\leq C(T,k)\frac{\log(n)^{2j}}{\sqrt{n}}$ for all $n \geq n_0(T,k)$. Redefining $n_0(T, k)$ as the maximum of the two completes the proof.
\end{proof}

\subsection{Concentration Bounds}

The primary ingredient of this section is the following concentration bound whose proof is deferred further to Appendix \ref{Concentration Proof}.
\begin{lm}\label{Concentration Bound}
    Let $F:\R^{n^2+(k+1)n} \to \R$ be a function with bounded Lipschitz norm on compact sets (independently of $n$ or $k$) and $|F(x)| \leq C(1+||x||)^d$ for some constants $C,d>0$ (again independent of $n$ or $k$). Then
    \begin{equation*}
    \begin{split}
        &\pr\left(\sup_{t \in [0,T]} |F(\Tilde{M}(t)/\sqrt{n},\theta_0(t),\dots,\theta_k(t))-\E F(\Tilde{M}(t)/\sqrt{n},\theta_0(t),\dots,\theta_k(t))| > u\right) \\
        &\hspace{7cm}\leq C e^{-c(T,k) \frac{n u^2}{\log(n)}}+e^{-c(T,k)\frac{n}{\log(n)}},
        \end{split}
    \end{equation*}
    where the constants $C$ and $c(T,k)$ do not depend on $n$. In particular, setting $u=\sqrt{w}\log(n)n^{-1/2}$, we have
    \begin{equation*}
        \pr\left(\sup_{t \in [0,T]} |F(\Tilde{M}(t)/\sqrt{n},\theta_0(t),\dots,\theta_k(t))-\E F(\Tilde{M}(t)/\sqrt{n},\theta_0(t),\dots,\theta_k(t))| > \frac{\sqrt{w}\log(n)}{n^{1/2}}\right) \leq C n^{-c(T,k)w},
    \end{equation*}
    which is summable for $w>c(T,k)^{-1}$. 
\end{lm}

The following lemma immediately follows from the previous lemma together with a union bound.
\begin{lm}\label{Concentration for specific functionals of theta}
    There is a $C=C(T,k)>0$ such that on a set $A_k$ with $\pr(A_k)>1-C'(T,k)n^{-2}$, the following holds. 
    \begin{enumerate}
        \item For any deterministic unit vector $v \in V_{k}:=\mathrm{Span}(\{e_j\}_{j=1}^k)$ and $x(t) \in \{\theta_j(t)\}_{j\leq k} \cup\{M(t)^{j_1}e_{j_2+2}\}_{0 \leq j_1 \leq k, 0 \leq j_2 \leq k-1}$, we have
    \begin{equation}
        \sup_{t \in [0,T]} |\la x(t),v \ra| \leq \frac{C\log(n)}{\sqrt{n}}.
        \label{Single Component bound}
    \end{equation}
    \item For $j, \ell =0,1,\dots,k$,
    \begin{equation}
        \sup_{t \in [0,T]} |\theta_j(t)^T\theta_{\ell}(t)-\delta_{j,\ell}| \leq \frac{C\log(n)}{\sqrt{n}}.
        \label{inner product concentration}
    \end{equation}
    \item 
    \begin{equation}
        \sup_{t \in [0,T]} \max(\|M(t)\|_{op},\|\tilde{M}(t)\|_{op}) \leq C\log(n).
    \end{equation}
    \end{enumerate}
\end{lm}

\subsection{$\vartheta_k(t) \approx \frac{\theta_k(t)}{\|\theta_{k-1}(t)\|}$: From Householder Reflectors to Gram-Schmidt}

Recall from Section~\ref{Householder and Tridiag Section}, with $M_0(t)=M(t)$, 
\begin{equation}
    \begin{split}
        a_{k+1}(t)&=e_{k+1}^T M_k(t)e_{k+1},\\
        \vartheta_{k}(t)&=\frac{1}{\sqrt{n}}M_k(t)e_{k+1}-\|\vartheta_{k-1}(t)\|e_k-\frac{1}{\sqrt{n}}a_{k+1}(t)e_{k+1},\\
        b_{k+1}(t) &= \|\vartheta_k(t)\|,\\
         Q_k(t)&=I-\frac{2(\vartheta_k(t)-\|\vartheta_k(t)\|e_{k+2})(\vartheta_k(t)-\|\vartheta_k(t)\|e_{k+2})^T}{\|\vartheta_k(t)-\|\vartheta_k(t)\|e_{k+2}\|^2},\\
        M_{k+1}(t) &= Q_k(t)M_k(t)Q_k(t),
    \end{split}
\end{equation}
and that all vectors in the equation display above lie in $\R^{n+1}$ and all matrices in $\R^{(n+1)^2}$. Moreover, recall that the first $k+1$ entries of $\vartheta_k(t)$ are identically $0$.
Finally recall the two vector spaces 
\begin{equation*}  
    \begin{split} ~&\Theta_k(t):=\mathrm{Span}(\{\theta_j(t)\}_{j\leq k}),\\
    &\mathcal{E}_k(t):= \mathrm{Span}(\{M(t)^{j_1}e_{j_2+2}\}_{0 \leq j_1 \leq k, 0 \leq j_2 \leq k-1}).
    \end{split}
\end{equation*}


To state the next result recall the definitions $\Tilde{M}(t)=(I-e_1e_1^T) M(t) (I-e_1e_1^T)$ and $\theta_k(t)=P_k(\Tilde{M}(t)/\sqrt{n})\vartheta_0(t)$ in Section~\ref{Approximate thetas}. The main result of this section is the following.
\begin{prop}
   On the event $A_k$ defined in Lemma \ref{Concentration for specific functionals of theta}, for $n$ sufficiently large, we have,
\begin{equation}\label{vartheta = theta + epsilon}
        \vartheta_k(t) = (c_k(t)+\mu_k(t))\theta_k(t)+\epsilon_k(t),
    \end{equation}
    with $c_k(t):=\frac{1}{\|\theta_{k-1}(t)\|}$, $\epsilon_k(t) \in \Theta_{k-1}(t)+ \mathcal{E}_k(t)$, $\sup_{t \in [0,T]}\|\epsilon_k(t)\| \leq C_{k,T} \frac{\log(n)}{\sqrt{n}}$, and $\sup_{t \in [0,T]} |\mu_k(t)| \leq C_{k,T}\frac{\log(n)^{2(k-1)}}{n}$.
\end{prop}

\begin{proof}
The proof is by the strong induction on $k$. The base case $k=0$ is trivial. For the inductive step, we start by proving that on the event $A_k$ defined in Lemma \ref{Concentration for specific functionals of theta}, for $n$ sufficiently large, we have,
\begin{equation}\label{Householder Product}
        \begin{split}
            &Q_0(t)Q_1(t)\cdots Q_{k-1}(t)\\
            &=I +\sum_{i < k, 2\leq j \leq k} \left(c_{i,j}(t) \vartheta_i(t)e_j^T+d_{i,j}(t) e_j\vartheta_i(t)^T\right)+\sum_{j,j' < k} f_{j,j'}(t) \vartheta_j(t)\vartheta_{j'}(t)^T,
        \end{split}
    \end{equation}
for some processes $c_{i,j}(t), d_{i,j}(t), f_{j,j'(t)}$, with $\max_{i,j,j'} \sup_{t \in [0,T]} (|c_{i,j}(t)|+ |d_{i,j}(t)|+|f_{j,j'(t)}|) \leq C_{k,T}$ for some constant $C_{k,T}>0$ depending only on $k$ and $T$.

Observe that by the inductive hypothesis and Lemma \ref{Concentration for specific functionals of theta}, it is clear that $|\vartheta_j(t)^T \vartheta_{j'}(t)-\delta_{j,j'}| \leq C_k \frac{\log(n)}{\sqrt{n}}$ and $|\vartheta_j(t)^T e_{j'}|\leq C_k \frac{\log(n)}{\sqrt{n}}$ for $j, j' \leq k$. As a result, we obtain
\begin{equation*}
   \left| \frac{2}{\|\vartheta_j(t)-\|\vartheta_j(t)\|e_{k+2}\|^2}\right| \leq 2+C_k\frac{\log(n)}{\sqrt{n}}.
\end{equation*}
Expanding the product $Q_0(t)Q_1(t)\cdots Q_{k}(t)$, an expansion of the form in the right hand side of \eqref{Householder Product} is clear. Moreover, the $c_{i,j}(t),d_{i,j}(t), f_{j,j'}(t)$ are all polynomials of $\frac{2}{||\vartheta_j(t)-||\vartheta_j(t)||e_{k+2}||^2}$, $e_a^T \vartheta_b(t)$ $\vartheta_a(t)^T \vartheta_b(t)$, and $\delta_a(t)$, for $a,b \leq k$, with number of terms and coefficients only depending on $k$. This completes the proof of \eqref{Householder Product}. The proof of the inductive step will involve repeated applications of \eqref{Householder Product} in tandem with the induction hypothesis and Lemma \ref{Concentration for specific functionals of theta}.

For the inductive step, we seek to prove  \eqref{vartheta = theta + epsilon} in the $k+1$ case. First, from the construction of the Householder matrix, $Q_k(t)e_{k+2}=\frac{\vartheta_k(t)}{||\vartheta_k(t)||}$. Once again recalling that $|\vartheta_j(t)^T \vartheta_{j'}(t)-\delta_{j,j'}| \leq C_{k,T} \frac{\log(n)}{\sqrt{n}}$ and $|\vartheta_j(t)^T e_{j'}|\leq C_{k,T} \frac{\log(n)}{\sqrt{n}}$ for $j, j' \leq k$, we have 

\begin{equation*}
    \begin{split}
        ~&Q_0(t)\cdots Q_{k-1}(t)Q_k(t)e_{k+2}\\
        &= \bigg(I+\sum_{i <k, j \leq k} (c_{i,j}(t) \vartheta_i(t)e_j^T+d_{i,j}(t) e_j\vartheta_i(t)^T)+\sum_{j,j' < k } f_{j,j'}(t) \vartheta_j(t)\vartheta_{j'}(t)^T\bigg) \frac{\vartheta_k(t)}{||\vartheta_k(t)||}\\
            & = \frac{\vartheta_k(t)}{||\vartheta_k(t)||}+\varepsilon_k(t),
    \end{split}
\end{equation*}
where $||\varepsilon_k(t)||\leq C_k\frac{\log(n)}{\sqrt{n}}$, and $\varepsilon_k(t) \in \mathrm{Span}\{\vartheta_j(t)\}_{j\leq k-1}+\mathrm{Span}\{e_j(t)\}_{2\leq j\leq k+1} \subset \Theta_{k-1}(t)+\mathcal{E}_{k}(t)$. Applying the induction hypothesis to the last equation display, we can write
\begin{equation*}
    Q_0(t)\cdots Q_{k-1}(t)Q_k(t)e_{k+2}  =\left(c_k(t)+\mu_k(t)\right)\frac{\theta_k(t)}{||\vartheta_{k}(t)||}+\varepsilon_k'(t),
\end{equation*}
where $||\varepsilon_k'(t)||\leq C_{k,T}\frac{\log(n)^k}{\sqrt{n}}$, and $\varepsilon_k'(t) \in \Theta_{k-1}(t)+\mathcal{E}_{k}(t)$. 

Using the three term recurrence for the Chebyshev polynomials \eqref{Three term reccurence} and the norm bound of Lemma \ref{Concentration for specific functionals of theta}, we have,
\begin{equation*}
    \begin{split}
        &\frac{1}{\sqrt{n}}\Tilde{M}(t)Q_0(t)\cdots Q_{k-1}(t)Q_k(t)e_{k+2} \\
        &= \frac{c_k(t)+\mu_k(t)}{||\vartheta_{k}(t)||}(\theta_{k+1}(t)+\theta_{k-1}(t))+ \varepsilon_{k+1}(t)\\
        &=\frac{c_k(t)+\mu_k(t)}{||\vartheta_{k}(t)||}(\theta_{k+1}(t)+\vartheta_{k-1}(t))+ \varepsilon_{k+1}'(t),
    \end{split}
\end{equation*}
where $||\varepsilon_{k+1}(t)||,||\varepsilon_{k+1}'(t)||\leq C_{k,T}\frac{\log(n)^{k+1}}{\sqrt{n}}$, and $\varepsilon_{k+1}(t), \varepsilon_{k+1}'(t) \in \Theta_{k}(t)+\Tilde{M}(t)\mathcal{E}_{k}(t) \subseteq \Theta_{k}(t)+\mathcal{E}_{k+1}(t)$. Then once again using the induction hypothesis, Lemma \ref{Concentration for specific functionals of theta}, and equation \eqref{Householder Product}, we get
\begin{equation*}
    \begin{split}
        &\frac{1}{\sqrt{n}}Q_{k-2}(t)\cdots Q_{0}(t)\Tilde{M}(t)Q_0(t)\cdots Q_{k-1}(t)Q_k(t)e_{k+2}\\
        &\hspace{2cm}= \frac{c_k(t)+\mu_k(t)}{||\vartheta_{k}(t)||}(\theta_{k+1}(t)+\vartheta_{k-1}(t))+ \varepsilon_{k+1}''(t),
    \end{split}
\end{equation*}
where $||\varepsilon_{k+1}''(t)||\leq C_{k,T}\frac{\log(n)^{k+1}}{\sqrt{n}}$, and $\varepsilon_{k+1}''(t) \in \Theta_{k}(t)+\mathcal{E}_{k}(t)+\Tilde{M}(t)\mathcal{E}_{k}(t) \subseteq \Theta_{k}(t)+\mathcal{E}_{k+1}(t)$. Recall that $Q_{k-1}(t)\vartheta_{k-1}(t) =||\vartheta_{k-1}(t)||e_{k+1}$, so applying $Q_{k-1}(t)$ on the left we have
\begin{equation*}
    \begin{split}
        &\frac{1}{\sqrt{n}}Q_{k-1}(t)\cdots Q_{0}(t)\Tilde{M}(t)Q_0(t) \cdots Q_{k-1}(t)Q_k(t)e_{k+2} \\
        &\hspace{2cm}= \frac{c_k(t)+\mu_k(t)}{||\vartheta_{k}(t)||}(\theta_{k+1}(t)+||\vartheta_{k-1}(t)||e_{k+1})+ \varepsilon_{k+1}'''(t),
    \end{split}
\end{equation*}
where $||\varepsilon_{k+1}'''(t)||\leq C_{k,T}\frac{\log(n)^{k+1}}{\sqrt{n}}$, and $\varepsilon_{k+1}'''(t) \in \Theta_{k}(t)+\mathcal{E}_{k+1}(t)$. Finally, applying $Q_k(t)$ on the left we obtain
\begin{equation*}
    \begin{split}
        &\frac{1}{\sqrt{n}}Q_k(t)Q_{k-1}(t)\cdots Q_{0}(t)\Tilde{M}(t)Q_0(t)\cdots Q_{k-1}(t)Q_k(t)e_{k+2} \\
        &\hspace{2cm} = \frac{c_k(t)+\mu_k(t)}{||\vartheta_{k}(t)||}(\theta_{k+1}(t)+||\vartheta_{k-1}(t)||e_{k+1})+ \varepsilon_{k+1}^{(4)}(t),
    \end{split}
\end{equation*}
with $||\varepsilon_{k+1}^{(4)}(t)||\leq C_{k,T}\frac{\log(n)}{\sqrt{n}}$, and $\varepsilon_{k+1}^{(4)}(t) \in \Theta_{k}(t)+\mathcal{E}_{k+1}(t)$. 

For the next step we estimate the error from $M(t)\approx \Tilde{M}(t)$. Using \eqref{Householder Product} and the fact that the first $k$ entries of $\vartheta_k(t)$ are zero by definition, we first see that
\begin{equation}\label{eq:Q's-fix-basis}
    Q_k(t)Q_{k-1}(t)\cdots Q_{0}(t)e_1=e_1.
\end{equation}
So plugging in the definition $\tilde{M}(t)=(I-e_1e_1^T)M(t)(I-e_1e_1^T)$ and applying \eqref{eq:Q's-fix-basis},
\begin{equation*}
    \begin{split}
        ~&\frac{1}{\sqrt{n}}Q_k(t)Q_{k-1}(t)\cdots Q_{0}(t)(M(t)-\tilde{M}(t))Q_0(t)\cdots Q_{k-1}(t)Q_k(t)\\
        &=\frac{1}{\sqrt{n}}Q_{k}(t)\cdots Q_{0}(t)\left(e_1 e_1^T M(t)+M(t) e_1e_1^T-e_1e_1^TM(t)e_1e_1^T\right)\\
        &Q_0(t)\cdots Q_{k-1}(t)Q_{k-1}(t)\\
        &=\frac{1}{\sqrt{n}}(e_1e_1^T M(t)Q_0(t)\cdots Q_{k-1}(t)Q_{k-1}(t)\\
        &+(e_1e_1^T M(t)Q_0(t)\cdots Q_{k-1}(t)Q_{k-1}(t))^T+ e_1e_1^T M(t)e_1 e_1^T) \\
        &=\frac{1}{\sqrt{n}}(e_1e_1^T M(t)+M(t)e_1e_1^T+e_1e_1^T M(t)e_1 e_1^T)+E_k(t)\\ 
        &= \frac{1}{\sqrt{n}}(M(t)-\tilde{M}(t))+E_k(t),
    \end{split}
\end{equation*}
where $\|E_k(t)\|\leq C_k\frac{\log(n)}{\sqrt{n}}$, and $\mathrm{Range}(E_k(t)) \subset \Theta_{k-1}(t)+\mathcal{E}_{k+1}(t)$. These properties of the error term in the above equation display, $E_k(t)$, follow from equation \eqref{Householder Product} and Lemma \ref{Concentration for specific functionals of theta} once again. In particular, we have
\begin{equation*}
    \|\frac{1}{\sqrt{n}}Q_{k}(t)\cdots Q_{0}(t)(M(t)-\Tilde{M}(t))Q_0(t)\cdots Q_{k-1}(t)Q_k(t)e_{k+2}\| \leq C_{k,T}\frac{\log(n)}{\sqrt{n}},
\end{equation*}
with $\frac{1}{\sqrt{n}}Q_{k}(t)\cdots Q_{0}(t)(M(t)-\Tilde{M}(t))Q_0(t)\cdots Q_{k-1}(t)Q_k(t)e_{k+2} \in \Theta_{k}(t)+\mathcal{E}_{k+1}(t)$. 

Combining everything so far, we have
\begin{equation*}
    \begin{split}
        \frac{1}{\sqrt{n}}M_{k+1}(t)e_{k+2}=\frac{c_k(t)+\mu_k(t)}{||\vartheta_{k}(t)||}\theta_{k+1}(t)+ \varepsilon_{k+1}^{(5)}(t),
    \end{split}
\end{equation*}
with $\|\varepsilon_{k+1}^{(5)}(t)\|\leq C_{k,T}\frac{\log(n)^{k+1}}{\sqrt{n}}$, and $\varepsilon_k^{(5)}(t) \in \Theta_{k}(t)+\mathcal{E}_{k}(t)+\Tilde{M}(t)\mathcal{E}_{k}(t)$. However, by the induction hypothesis we have
\begin{equation}\label{Norm of theta computation}
    \begin{split}
        ||\vartheta_k(t)||^2&=(c_k(t)+\mu_k(t))^2||\theta_k(t)||^2+2(c_k(t)+\mu_k(t))\epsilon_k(t)^T\theta_k(t)+||\epsilon_k(t)||^2\\
        &=c_k(t)^2||\theta_k(t)||^2+O(\log(n)^{2k}/n),
    \end{split}
\end{equation}
and thus similarly $\frac{1}{\|\vartheta_k(t)\|}=\frac{1}{c_k(t)\|\theta_k(t)\|}+O(\log(n)^{2k}/n)$. Finally we deduce that 
\begin{equation*}
    \begin{split}
        \vartheta_{k+1}(t)&=\frac{1}{\sqrt{n}}M_{k+1}(t)e_{k+2}-\|\vartheta_{k-1}(t)\|e_{k+1}-\frac{a_{k+1}(t)}{\sqrt{n}}\\
        &=\bigg(\frac{1}{\|\theta_{k}(t)\|}+O(\log(n)^{2k}/n)\bigg)(\theta_{k+1}(t)+\|\vartheta_{k-1}(t)\|e_{k+1}+ \epsilon_{k+1}'(t))\\
        &\hspace{7cm}-\|\vartheta_{k-1}(t)\|e_{k+1}-\frac{a_{k+1}(t)}{\sqrt{n}}\\
        &=c_{k+1}(t)\theta_{k+1}(t)+\mu_{k+1}(t)\theta_{k+1}(t)+\epsilon_{k+1}(t).
    \end{split}
\end{equation*}
With $c_{k+1}(t),\mu_{k+1}(t),\epsilon_{k+1}(t)$ all satisfying the desired properties. Here we used the fact that $|a_{k+1}(t)| \leq C_{k,T}$ which is easy to prove by the induction hypothesis and Lemma \ref{Concentration for specific functionals of theta}. This completes the proof. 
\end{proof}

We are now almost in the position to analyze the tridiagonal entries $a_j(t)$ and $b_j(t)$. But, let us first state a simple corollary of the last proposition.
\begin{cor}\label{vartheta expansion and analysis}
    There is a $C=C(T,k)>0$ such that on a set $A_k$ with $\pr(A_k)>1-c(T,k)n^{-2}$ the following holds.
    \begin{enumerate}
        \item For any deterministic unit vector $v \in V_{k}=\mathrm{Span}(\{e_j\}_{j=1}^k)$ and $x(t) \in \{\vartheta_j(t)\}_{j\leq k} \cup\{M(t)^{j_1}e_{j_2+2}\}_{0 \leq j_1 \leq k, 0 \leq j_2 \leq k-1}$
    \begin{equation}
        \sup_{t \in [0,T]} |\la x(t),v \ra| \leq \frac{C\log(n)^k}{\sqrt{n}}.
        \label{Single Component bound var}
    \end{equation}
    In particular, the inequality also holds for $x(t)=\vartheta_j(t)$ for $j=1,\dots ,k$.
    \item For $j, \ell =0,1,\dots,k$,
    \begin{equation}
        \sup_{t \in [0,T]} |\vartheta_j(t)^T\vartheta_{\ell}(t)-\delta_{j,\ell}| \leq \frac{C\log(n)}{\sqrt{n}}.
        \label{inner product concentration var}
    \end{equation}
    \item For $j, \ell =0,1,\dots ,k$, we have
    \begin{equation}\label{Robust vartheta expansion}
        \begin{split}
            ~&\vartheta_j(t)=(c_j(t)+\mu_j(t))\theta_j(t)+\\
            &\bigg( \vartheta_j(t)^T\frac{\theta_{j-1}(t)}{\|\theta_{j-1}(t)\|^2}-c_j(t)\theta_j(t)^T\theta_{j-1}(t)+\tilde{\mu}_j(t)\bigg)\theta_{j-1}(t)+\tilde{\epsilon}_{j}(t)
        \end{split}
    \end{equation}
     with $c_j(t):=\frac{1}{\|\theta_{j-1}(t)\|}$, $\tilde{\epsilon}_{j}(t) \in \Theta_{j-2}(t)+ \mathcal{E}_j(t)$, $\sup_{t \in [0,T]}||\tilde{\epsilon}_{j}(t)|| \leq C_{k,T} \frac{\log(n)^{k}}{\sqrt{n}}$, and $\sup_{t \in [0,T]} (|\mu_j(t)|+|\tilde{\mu}_j(t)|) \leq C_{k,T}\frac{\log(n)^{2(k-1)}}{n}$.
    \end{enumerate} 
\end{cor}

\begin{proof}
    The first two statements above are immediate corollaries of \eqref{vartheta = theta + epsilon} and Lemma~\ref{Concentration for specific functionals of theta}. For the third statement, by \eqref{vartheta = theta + epsilon}, we must be able to express
    \begin{equation}\label{Undetermined vartheta expansion}
        \vartheta_j(t)=(c_j(t)+\mu_j(t))\theta_j(t)+\Tilde{c}_{j}(t)\theta_{j-1}(t)+\Tilde{\epsilon}_{j}(t),
    \end{equation}
    with some coefficient $\Tilde{c}_{j}(t)$, and some $\Tilde{\epsilon}_{j}(t) \in \Theta_{j-2}(t)+ \mathcal{E}_j(t)$, with 
    $\sup_{t \in [0,T]}||\tilde{\epsilon}_{j}(t)|| \leq C_{k,T} \frac{\log(n)^k}{\sqrt{n}}$ where we have implicitly used the approximate orthogonality between $\theta_{j-1}(t)$ and $\Theta_{j-2}(t)+ \mathcal{E}_j(t)$ from Lemma \ref{Concentration for specific functionals of theta}. To solve for $\tilde{c}_j(t)$, simply take the inner product of both sides of \eqref{Undetermined vartheta expansion} with $\theta_{j-1}(t)$ to obtain
    \begin{equation}
        \vartheta_j(t)^T\theta_{j-1}(t)=(c_j(t)+\mu_j(t))\theta_j(t)^T\theta_{j-1}(t)+\Tilde{c}_j(t)\|\theta_{j-1}(t)\|^2+\nu_j(t),
    \end{equation}
    where $\sup_{t \in [0,T]} |\nu_j(t)|\leq C_{k,T}\frac{\log(n)^k}{n}$. Rearranging the equation above yields the desired result.
\end{proof}

Now we are ready to approximate $a_k(t)$ and $b_k(t)$.
\begin{prop}
    On the event $A_k$ defined in Lemma \ref{Concentration for specific functionals of theta}, for $n$ sufficiently large, we have
    \begin{equation}\label{a_k approx equation}
        \sup_{t \in [0,T]}|a_{k+1}(t)-\sqrt{n}(\theta_{k}(t)-\theta_{k-2}(t))^T\theta_{k-1}(t)| \leq C_{k,T}\frac{\log(n)^k}{\sqrt{n}},
    \end{equation}
    and
    \begin{equation}\label{b_k approx equation}
        \sup_{t \in [0,T]}\left|b_{k+1}(t)^2-\frac{\|\theta_k(t)\|^2}{\|\theta_{k-1}(t)\|^2}\right| \leq C_{k,T}\frac{\log(n)^{2(k-1)}}{n}.
    \end{equation}
\end{prop}

\begin{proof}
    First, the proof of \eqref{b_k approx equation} is a straightforward consequence of \eqref{vartheta = theta + epsilon}, and was already given in \eqref{Norm of theta computation}.

    \par To prove \eqref{a_k approx equation}, we first recall that $Q_{k-1}(t)$ is constructed to be the Householder matrix with the property that $Q_{k-1}(t)e_{k+1}=\frac{\vartheta_{k-1}(t)}{\|\vartheta_{k-1}(t)\|}$. Next we compute
    \begin{equation}\label{eq:Q(k-2)-comp}
        \begin{split}
        Q_{k-2}(t)\vartheta_{k-1}(t)&=\vartheta_{k-1}(t)-\frac{2(\vartheta_{k-2}(t)-\|\vartheta_{k-2}(t)\|e_{k})^T \vartheta_{k-1}(t)}{\|\vartheta_{k-2}(t)-\|\vartheta_{k-2}(t)\|e_{k}\|^2} (\vartheta_{k-2}(t)-e_k)\\
        &=\vartheta_{k-1}(t)-\frac{2\vartheta_{k-2}(t)^T \vartheta_{k-1}(t)}{\|\vartheta_{k-2}(t)-\|\vartheta_{k-2}(t)\|e_{k}\|^2} \vartheta_{k-2}(t)+\delta(t),
        \end{split}
    \end{equation}
     where using the fact that the first $k$ entries of $\vartheta_{k-1}(t)$ are $0$,
     \begin{equation*}
     \begin{split}
         \delta(t)&= \frac{2(\vartheta_{k-2}(t)-\|\vartheta_{k-2}(t)\|e_{k})^T \vartheta_{k-1}(t)}{\|\vartheta_{k-2}(t)-\|\vartheta_{k-2}(t)\|e_{k}\|^2}e_{k}-\frac{2e_k^T \vartheta_{k-1}(t)}{\|\vartheta_{k-2}(t)-\|\vartheta_{k-2}(t)\|e_{k}\|^2}\vartheta_{k-2}(t)\\
         &= \frac{2(\vartheta_{k-2}(t)-\|\vartheta_{k-2}(t)\|e_{k})^T \vartheta_{k-1}(t)}{\|\vartheta_{k-2}(t)-\|\vartheta_{k-2}(t)\|e_{k}\|^2}e_{k} \in \mathcal{E}_{k-2}
     \end{split}
     \end{equation*}
     and $\|\delta(t)\| \leq C_{k,T} \frac{\log(n)^{k-1}}{\sqrt{n}}$ by Corollary \ref{vartheta expansion and analysis}. Combining this computation \eqref{eq:Q(k-2)-comp} with the expansion of the product of the Householder matrices \eqref{Householder Product}, starting from the definition of $a_{k+1}(t)$, we obtain,
    \begin{equation*}
        \begin{split}
            a_{k+1}(t)&=e_{k+1}^T M_k(t) e_{k+1}\\
            &=e_{k+1}^T Q_{k-1}(t)Q_{k-2}(t)\cdots Q_0(t)M(t)Q_0(t)\cdots Q_{k-2}(t)Q_{k-1}(t)e_{k+1}\\
            &=\frac{\vartheta_{k-1}(t)^T}{\|\vartheta_{k-1}(t)\|}Q_{k-2}(t)\cdots Q_0(t)M(t)Q_0(t)\cdots Q_{k-2}(t)\frac{\vartheta_{k-1}(t)}{\|\vartheta_{k-1}(t)\|}\\
            &=\bigg(\vartheta_{k-1}(t)-\frac{2\vartheta_{k-2}(t)^T \vartheta_{k-1}(t)}{\|\vartheta_{k-2}(t)-\|\vartheta_{k-2}(t)\|e_{k}\|^2} \vartheta_{k-2}(t)+\delta(t)\bigg)^TQ_{k-3}(t)\cdots Q_{0}(t)\\
            &M(t)Q_{0}(t)\cdots Q_{k-3}(t)\bigg(\vartheta_{k-1}(t)-\frac{2\vartheta_{k-2}(t)^T \vartheta_{k-1}(t)}{\|\vartheta_{k-2}(t)-\|\vartheta_{k-2}(t)\|e_{k}\|^2} \vartheta_{k-2}(t)+\delta(t)\bigg)\\
            &= \bigg(\frac{\vartheta_{k-1}(t)}{\|\vartheta_{k-1}(t)\|}-\frac{2\vartheta_{k-2}(t)^T \vartheta_{k-1}(t)}{\|\vartheta_{k-1}(t)\|\|\vartheta_{k-2}(t)-e_{k}\|^2} \vartheta_{k-2}(t)+\delta'(t)\bigg)^T \\
            &M(t) \bigg(\frac{\vartheta_{k-1}(t)}{\|\vartheta_{k-1}(t)\|}-\frac{2\vartheta_{k-2}(t)^T \vartheta_{k-1}(t)}{\|\vartheta_{k-1}(t)\|
            |\vartheta_{k-2}(t)-e_{k}\|^2} \vartheta_{k-2}(t)+\delta'(t)\bigg) ,
        \end{split}
    \end{equation*}
    where $\delta'(t) \in \Theta_{k-3}(t)+\mathcal{E}_{k-2}(t)$, with $\|\delta'(t)\| \leq C_{k,T} \frac{\log(n)}{\sqrt{n}}$. Next, by Corollary \ref{vartheta expansion and analysis} and in particular \eqref{Robust vartheta expansion}, we have,
    \begin{equation*}
        \begin{split}
            &\frac{\vartheta_{k-1}(t)}{\|\vartheta_{k-1}(t)\|}-\frac{2\vartheta_{k-2}(t)^T \vartheta_{k-1}(t)}{\|\vartheta_{k-1}(t)\|\cdot\|\vartheta_{k-2}(t)-e_{k}\|^2} \vartheta_{k-2}(t)+\delta'(t) \\
            &= (\theta_{k-1}(t)+(\vartheta_{k-1}(t)^T\theta_{k-2}(t)-\theta_{k-1}(t)^T\theta_{k-2}(t))\theta_{k-2}(t)-\theta_{k-2}(t)^T \vartheta_{k-1}(t) \theta_{k-2}(t)+\delta''(t)+\delta'''(t)\\
            &=\theta_{k-1}(t)-(\theta_{k-1}(t)^T\theta_{k-2}(t))\theta_{k-2}(t)+\delta''(t)+\delta'''(t),
        \end{split}
    \end{equation*}   
    with $\delta''(t) \in \mathrm{Span}\{\theta_{k-1}(t),\theta_{k-2}(t)\}$, $\delta'''(t) \in \Theta_{k-3}(t)+\mathcal{E}_{k-1}(t)$, along with the bounds $\|\delta''(t)\| \leq C_{k,T} \frac{\log(n)^{2(k-2)}}{n}$, and $||\delta'''(t)|| \leq C_k \frac{\log(n)^{k-1}}{\sqrt{n}}$. 
    
    Combining everything so far and applying the three-term recurrence and Lemma~\ref{Concentration for specific functionals of theta}, we have
    \begin{equation*}
        \begin{split}
             a_{k+1}(t) &=(\theta_{k-1}(t)-\theta_{k-2}(t)^T \theta_{k-1}(t) \theta_{k-2}(t)+\delta''(t)+\delta'''(t))^T\\
             &M(t)(\theta_{k-1}(t)-\theta_{k-2}(t)^T \theta_{k-1}(t) \theta_{k-2}(t)+\delta''(t)+\delta'''(t))\\
             &=(\theta_{k-1}(t)-\theta_{k-2}(t)^T \theta_{k-1}(t) \theta_{k-2}(t)+\delta''(t)+\delta'''(t))^T\\
             &\Tilde{M}(t)(\theta_{k-1}(t)-\theta_{k-2}(t)^T \theta_{k-1}(t) \theta_{k-2}(t)+\delta''(t)+\delta'''(t))\\
             &+(\theta_{k-1}(t)-\theta_{k-2}(t)^T \theta_{k-1}(t) \theta_{k-2}(t)+\delta''(t)+\delta'''(t))^T\\
             &(M-\Tilde{M}(t))(\theta_{k-1}(t)-\theta_{k-2}(t)^T \theta_{k-1}(t) \theta_{k-2}(t)+\delta''(t)+\delta'''(t))\\
             &=\sqrt{n}(\theta_{k-1}(t)-\theta_{k-2}(t)^T \theta_{k-1}(t) \theta_{k-2}(t)+\delta''(t)+\delta'''(t))^T((1+\mu(t))\theta_{k}(t)+\theta_{k-2}(t)+\delta^{(4)}(t))\\
             &+(\theta_{k-1}(t)-\theta_{k-2}(t)^T \theta_{k-1}(t) \theta_{k-2}(t)+\delta''(t)+\delta'''(t))^T\\
             &(M(t)-\Tilde{M}(t))(\theta_{k-1}(t)-\theta_{k-2}(t)^T \theta_{k-1}(t) \theta_{k-2}(t)+\delta''(t)+\delta'''(t))\\
             &= \sqrt{n}(\theta_{k}(t)-\theta_{k-2}(t))^T\theta_{k-1}(t)+O(\log(n)^k/\sqrt{n}),
        \end{split}
    \end{equation*}
    where $|\mu(t)|=O(\frac{\log(n)^{2(k-1)}}{n})$, $\delta^{(4)}(t) \in \Theta_{k-1}(t)+\mathcal{E}_{k+1}(t)$, and $\|\delta^{(4)}(t)\|=O(\frac{\log(n)^k}{\sqrt{n}})$. Here we have used the fact that 
    \begin{equation*}
    \begin{split}
        &(\theta_{k-1}(t)-\theta_{k-2}(t)^T \theta_{k-1}(t) \theta_{k-2}(t)+\delta''(t)+\delta'''(t))^T(M(t)-\Tilde{M}(t))\\ 
        &(\theta_{k-1}(t)-\theta_{k-2}(t)^T \theta_{k-1}(t) \theta_{k-2}(t)+\delta''(t)+\delta'''(t))=O(\log(n)/\sqrt{n}),
    \end{split}
    \end{equation*} 
    which follows from an identical calculation in the proof of \eqref{vartheta = theta + epsilon}.
\end{proof}

Finally, as a corollary we can prove Proposition \ref{Quantitative Approximation Probability }. Namely, on $A_k$,
\begin{equation}\label{a_k approx}
    \begin{split}
        \sup_{t\in [0,T]}|a_{k+1}(t)-\sqrt{n}\theta(t)^T P_{2k-1}(\Tilde{M}(t)/\sqrt{n})\theta(t)| \leq C_k \frac{\log(n)}{\sqrt{n}},
    \end{split}
\end{equation}
and
\begin{equation}\label{b_k approx}
    \begin{split}
        \sup_{t\in [0,T]}|\frac{1}{\sqrt{n}}(b_{k+1}(t)^2-n)-\sqrt{n}\theta(t)^T P_{2k}(\Tilde{M}(t)/\sqrt{n})\theta(t)| \leq C_k \frac{\log(n)}{\sqrt{n}},
    \end{split}
\end{equation}
holds for $k \geq 1$. These are equivalent to Proposition~\ref{Quantitative Approximation Probability }.

\begin{proof}[Proof of Proposition~\ref{Quantitative Approximation Probability }]
    Recall the Chebyshev polynomial identity from Section~\ref{Chebyshev Polynomial Section},
    \begin{equation*}
        P_{j}(x)P_k(x) = \sum_{\ell \leq [(j+k)/2]} P_{j+k-\ell}(x).
    \end{equation*}
    As a direct corollary, $P_{k}(x)^2-P_{k-1}(x)^2=P_{2k}(x)$, and $(P_{k}(x)-P_{k-2}(x))P_{k-1}(x)=P_{2k-1}(x)$, and therefore,
    \begin{equation*}
        \begin{split}
            \frac{||\theta_k(t)||^2}{||\theta_{k-1}(t)||^2}-1 &=\frac{1}{||\theta_{k-1}(t)||^2}(||\theta_k(t)||^2-||\theta_{k-1}(t)||^2) \\
            &=(1+O(\log(n)/\sqrt{n}))\theta(t)^T \left(P_{k}(\Tilde{M}(t)/\sqrt{n})^2-P_{k-1}(\Tilde{M}(t)/\sqrt{n})^2\right)\theta(t)\\
            &=(1+O(\log(n)/\sqrt{n}))\theta(t)^T P_{2k}(\Tilde{M}(t)/\sqrt{n})\theta(t),
        \end{split}
    \end{equation*}
    and similarly
    \begin{equation*}
        \begin{split}
            (\theta_{k}(t)-\theta_{k-2}(t))^T\theta_{k-1}(t)&=\theta(t)^T(P_{k}(\Tilde{M}(t)/\sqrt{n})-P_{k-2}(\Tilde{M}(t)/\sqrt{n}))P_{k-1}(\Tilde{M}(t)/\sqrt{n})\theta(t)\\
            &=\theta(t)^T P_{2k-1}(\Tilde{M}(t)/\sqrt{n})\theta(t).
        \end{split}
    \end{equation*}
    The proof now follows from \eqref{a_k approx equation} and \eqref{b_k approx equation}.
\end{proof}

\section{Proof of Theorem \ref{Weak Convergence of Approximate Tridiagonal Entries}}\label{Moment Method Section}

\subsection{Orthogonality Computation}

Recall in Section~\ref{Chebyshev Polynomial Section}, we took $P_k(x)$ to be the orthonormal polynomials with respect to the semicircular measure $\frac{1}{2\pi}\mathbbm{1}_{|x|\leq 2} \sqrt{4-x^2} dx$, so it immediately follows from the Wigner semicircle law that $\frac{1}{n}\E[\mathrm{Tr}(P_j(\Tilde{M}(t)/\sqrt{n})P_k(\Tilde{M}(t)/\sqrt{n}))]\to \delta_{j,k}$ as $n \to \infty$. Using a more explicit combinatorial argument, we see this computation generalizes to two distinct times. This is the content of the following lemma.

\begin{lm}\label{Semicircular Cov}
We have
\begin{equation*}
    \frac{1}{n}\E[\mathrm{Tr}(P_j(\Tilde{M}(t_1)/\sqrt{n})P_k(\Tilde{M}(t_2)/\sqrt{n}))]=\delta_{j,k} e^{-k|t_1-t_2|}+O(1/n).
\end{equation*}
\end{lm}

In what follows we will use the definitions and theorem from Section~\ref{Non-crossing Partitions}. To begin with, let $j+k=2m$ and $\Tilde{M}(t)$ be defined as in Section~\ref{Approximate thetas}. We can compute directly from Theorem~\ref{Semicircular system asymptotics} that,
\begin{equation*}
    \begin{split}
        &\E\left[\frac{1}{n}\mathrm{Tr}((\Tilde{M}(t_1)/\sqrt{n})^j(\Tilde{M}(t_2)/\sqrt{n})^k)\right]\\
        &=\sum_{\pi \in NC[2m]} \prod_{(i_1,i_2) \in \pi} (\delta_{\{i_1,i_2 \leq j\} \cup \{i_1,i_2 > j\}}+\delta_{\{i_1 \leq j < i_2\} \cup \{i_2 \leq j < i_1\}} e^{-|t_1-t_2|})+O(1/n)\\
        &= \sum_{\ell=0}^{j \wedge k} C_{n,j,\ell}e^{-\ell |t_1-t_2|}+O(1/n),
    \end{split}
\end{equation*}
where $C_{m,j,\ell}$ is the number of $\pi\in NC_2[2m]$ with $\pi$ having exactly $\ell$ pairs with one element from $\{1, \dots, j\}$ and one from $\{j+1, \dots, 2m\}$. For convenience we will also define
\begin{equation}
    \kappa^{j,k}_{\pi}:=\prod_{(i_1,i_2) \in \pi} (\delta_{\{i_1,i_2 \leq j\} \cup \{i_1,i_2 > j\}}+\delta_{\{i_1 \leq j < i_2\} \cup \{i_2 \leq j < i_1\}} e^{-|t_1-t_2|}),
\end{equation}
for any $\pi \in NC_2[j+k]$. Now we are ready to prove Lemma \ref{Semicircular Cov}.

\begin{proof}\textbf{(Proof of Lemma~\ref{Semicircular Cov}})
Since the $P_j$ are monic polynomials, it suffices to prove that for $j\leq k$, 
\begin{equation*}
    \frac{1}{n}\E[\mathrm{Tr}((\Tilde{M}(t_1)/\sqrt{n})^j P_k(\Tilde{M}(t_2)/\sqrt{n}))]=\delta_{j,k} e^{-k|t_1-t_2|}+O(1/n),
\end{equation*}
and for $k\leq j$, 
\begin{equation*}
\frac{1}{n}\E[\mathrm{Tr}(P_j(\Tilde{M}(t_1)/\sqrt{n}) (\Tilde{M}(t_2)/\sqrt{n})^k)]=\delta_{j,k} e^{-k|t_1-t_2|}+O(1/n).    
\end{equation*}
We will just prove the first equation since the second equation can be proved identically. 

Divide $V=\{1,\dots,j+k\}$ into $V_1=\{1,\dots,j\}$ and $V_2=\{j+1,\dots,j+k\}$. There is exactly one non-crossing pair partition which pairs every element of $V_1$ to an element in $V_2$ if $j=k$, and $0$ such partitions otherwise. Also observe for any $\pi \in NC_2[V]$, if two numbers in $V_2$ get paired together, there must be at least one adjacent pair of numbers in $V_2$ paired. Thus we see that
\begin{equation*}
    \begin{split}
        \frac{1}{n}&\E[\mathrm{Tr}((\Tilde{M}(t_1)/\sqrt{n})^j P_k(\Tilde{M}(t_2)/\sqrt{n}))]\\
        &= \sum_{k'=0}^{[k/2]}(-1)^{k'} \binom{k-k'}{k'} \frac{1}{n}\E\left[\mathrm{Tr}((M(t_1)/\sqrt{n})^j (M(t_2)/\sqrt{n})^{k-2k'})\right]\\
        &=\sum_{k'=0}^{[k/2]}(-1)^{k'} \binom{k-k'}{k'} \sum_{\pi \in NC_2[j+k-2k']} \kappa^{j,k-2k'}_{\pi}+O(1/n)\\
        &= \sum_{k'=0}^{[k/2]}(-1)^{k'} \sum_{{\{i'_1,i'_1+1\},\dots,\{i'_{k'} i'_{k'}+1\} \in V_2}} \sum_{\pi' \in NC_2[j+k]: \{i'_1,i'_1+1\},\dots,\{i'_{k'},i'_{k'}+1\} \in \pi} \kappa_{\pi'}^{j,k}+O(1/n)\\
        &= \sum_{\pi \in NC_2[j+k], \text{ all pairs containing one element in } V_1 \text{ and one in } V_2} \kappa_{\pi}^{j,k}+O(1/n)\\
        &=\delta_{j,k} e^{-k|t_1-t_2|}+O(1/n)
    \end{split}
\end{equation*}
In the calculations above we used the explicit formula for $P_k(x)$, \eqref{P_k(x) explicit formula} and the observation that there are exactly $\binom{k-k'}{k'}$ choices for inserting positions of the $k'$ adjacent pairs in $V_2$ described by the second sum. The second to last equality follows from the inclusion--exclusion principle, and finally the last equality follows from the earlier observation that the last sum only contains one partition if $j=k$ and none otherwise. This completes the proof.    
\end{proof}

\subsection{Analysis of Moments and Tightness}

The main result of this section is the convergence of moments for $\Tilde{E}_n(t) \to E(t)$.

\begin{prop}\label{Convergence of Moments}
For any fixed $j_1,j_2$, $k_1,\dots,k_{j_1}$, and $k'_1,\dots,k_{j_2}'$ we have
\begin{equation*}
    \begin{split}
        \E&\left[\Tilde{a}_{k_1}(t_1)\cdots \Tilde{a}_{k_{j_1}}(t_{j_1})\frac{1}{\sqrt{n}}\left(\Tilde{b}_{k'_1}(t_1')^2-n\right)\cdots\frac{1}{\sqrt{n}}\left(\Tilde{b}_{k'_{j_2}}(t_{j_2}')^2-n\right)\right]\\
        &=\E\left[A_{k_1}(t)\cdots A_{k_{j_1}}(t)B_{k_1}(t)\cdots B_{k_{j_2}}(t)\right]+O(1/n)\\
        & =\bigg(2^{j_1/2}\sum_{\pi \in \mathcal{P}_2(j_1)} \prod_{(i_1,i_2) \in \pi}e^{-(2k_{i_1}-1)|t_{i_1}-t_{i_2}|}\delta_{k_{i_1},k_{i_2}} \bigg)\\
        & \hspace{4cm}\times\bigg(2^{j_2/2}\sum_{\pi \in \mathcal{P}_2(j_2)} \prod_{(i_1,i_2) \in \pi}e^{-2k'_{i_1}|t_{i_1}'-t_{i_2}'|}\delta_{k'_{i_1},k'_{i_2}} \bigg)+O(1/\sqrt{n}).
    \end{split}
\end{equation*}
\end{prop}

A standard argument using the Weierstrass approximation theorem and uniform boundedness of moments implies weak convergence as an immediate corollary of Proposition~\ref{Convergence of Moments}.
\begin{cor}\label{Finite Dimensional Distributions}
For any fixed $j \in \mathbb{N}$ and times $t_1,\dots ,t_j \in [0,T]$, we have
\begin{equation*}
    \begin{split}
        (\Tilde{E}_n(t_1),\dots,\Tilde{E}_n(t_j)) \to (E(t_1),\dots ,E(t_j)),
    \end{split}
\end{equation*}
in distribution.
\end{cor}

To prove Proposition~\ref{Convergence of Moments}, we will need the following algorithm for turning certain pairings into permutations. We will denote by $\mathcal{S}(J)$ the set of permutations of the elements in the set $J$. Given an index set $J \subset \mathbb{N}$, the algorithm below defines a function $\mathrm{Perm}:\mathcal{P}_2(\{1,2\} \times J) \to  \mathcal{S}(J)$.    
\begin{alg}\,
    \begin{enumerate}
        \item Set $\mathcal{J}:= \{1,2\} \times  J $, and $i,i'=0$.
        \item Let $j_i^{(i')} := \min \{j \in J: (1,j) \in \mathcal{J}  \}$ and $a_i^{(i')}=1$. 
        \item Let $\tau$ be the function which flips the first coordinate. Define $(a_i^{(i'+1)}j_i^{(i'+1)}):=\tau \circ \pi (a_i^{(i')}j_i^{(i')})$. Moreover remove $(a_i^{(i'+1)}j_i^{(i'+1)})$ and $\pi(a_i^{(i'+1)}j_i^{(i'+1)})$ from $\mathcal{J}$.
        \item If $j_i^{(i'+1)}=j_i^{(0)}$ define the cycle $\sigma_i = (j_i^{(0)}\cdots j_i^{(i')})$. Moreover, if $\mathcal{J}=\emptyset$ move to the next step, otherwise if $\mathcal{J}\neq \emptyset$ set $i = i+1$, $i'=0$ and go back to step 2. Else, if $j_i^{(i'+1)}\neq j_i^{(0)}$ let $i' =i'+1$ and go back to step 2.
        \item Set $\sigma = \mathrm{Perm}(\pi):= \sigma_1\cdots\sigma_i$
    \end{enumerate}
\end{alg}

We provide a helpful example. 
\begin{ex}
    Let $J=\{1,2,3,4\}$ and consider the two pairings
    \begin{equation*}
        \begin{split}
            \pi = \{\{(1,1),(1,3)\},\{(2,3),(2,2)\}, \{(1,2),(1,4)\},\{(2,4),(2,1)\}\}\\
            \pi' = \{\{(1,1),(1,2)\},\{(2,2),(2,1)\}, \{(1,3),(1,4)\},\{(2,4),(2,3)\}\}.
        \end{split}
    \end{equation*}
    Then, in cycle notation $\mathrm{Perm}(\pi) = (1 3 2 4)$ and $\mathrm{Perm}(\pi') = (1 2 ) (3 4)$. 
\end{ex}

The purpose of the above function $\mathrm{Perm}$ is to introduce the following Lemma. 
\begin{lm}\label{Perm Expansion}
    For deterministic matrices $A_1,\dots ,A_j$, we have
    \begin{equation*}
        \E\left[\prod_{i=1}^{j}\theta(t_i)^T A_i \theta(t_i)\right] = \sum_{\pi \in \mathcal{P}_2(\{1,2\} \times [j])} \left(\prod_{i \in [j]} \frac{e^{-|t_i-t_{\sigma(i)}|}}{n}\right)\prod_{i=1}^{l}\mathrm{Tr}\left(A_{\sigma_i^{(1)}}\cdots A_{\sigma_i^{(|\sigma_i|)}}\right),
    \end{equation*}
    where $\sigma = \sigma_1\cdots \sigma_l = \mathrm{Perm}(\pi)$ in the sum above, and each cycle in the disjoint cycle decomposition can be written as $\sigma_i = (\sigma_i^{(0)}\cdots \sigma_i^{(|\sigma_i|)})$.
\end{lm}

\begin{proof}
    This follows by carefully tracking indices after applying Wick's theorem with the covariances, $\E[(\theta(t))_i(\theta(s))_j] = \frac{e^{-|t-s|}}{n} \delta_{i,j}$.
\end{proof}

Next we would like to precisely count terms which are repeated. 
\begin{lm}\label{Multiplicity of Perm}
    For $J \subset \mathbb{N}$, and each permutation $\sigma \in \mathcal{S}(J)$ with disjoint cycle decomposition $\sigma = \sigma_1\cdots\sigma_l$, there are exactly $2^{\sum_{i=1}^{l}(|\sigma_i|-1)}$ pairings $\pi \in \mathcal{P}_2(\{0,1\} \times J)$ such that $\mathrm{Perm}(\pi) = \sigma$.
\end{lm}

\begin{proof}
    The freedom in the algorithm is the choice of $a_1,\dots ,a_{|\sigma_i|-1} \in \{1,2\}$ for each cycle which gives $2^{|\sigma_i|-1}$ choices per cycle $\sigma_i$.
\end{proof}

We are now ready to prove Proposition \ref{Convergence of Moments}.

\begin{proof}\textbf{(Proof of Proposition~\ref{Convergence of Moments})}
For convenience we assume that $k_1,\dots ,k_{j_1}>1$, otherwise simply observe $\Tilde{a}_1(t)$ is independent of $\Tilde{a}_j(t),\Tilde{b}_j(t)$ for $j>1$. 

Now recall from Definition~\ref{Def of approximate Tridiagonal Entries} that $\Tilde{a}_{j+1}(t) = \sqrt{n} \theta(t)^T P_{2j-1}(\Tilde{M}(t)/\sqrt{n})\theta(t)$ for $j>0$, and $\frac{1}{\sqrt{n}}(\Tilde{b}_{j+1}(t)^2-n)=\sqrt{n} \theta(t)^T P_{2j}(\Tilde{M}(t)/\sqrt{n})\theta(t)$ for all $j \geq 0$. Taking the conditional expectation with respect to $(\Tilde{M}(t))_{t \in [0,T]}$ we can apply the previous two lemmas to obtain
\begin{equation*}
    \begin{split}
        \E&\bigg[ \bigg( \prod_{i=1}^{j_1} (\theta(t_i)^T P_{2k_i-1}(\Tilde{M}(t_i)/\sqrt{n})\theta(t_i))\bigg)\bigg( \prod_{i=1}^{j_1} (\theta(t_i')^T P_{2k_i'}(\Tilde{M}(t_i')/\sqrt{n})\theta(t_i))\bigg)\bigg] \\
        &=\sum_{\sigma=\sigma_1\dots \sigma_{\ell} \in \mathcal{S}([j_1+j_2])}\prod_{i \in [j_1+j_2]} \frac{e^{-|\tau_{i}-\tau_{\sigma(i)}|}}{n} \E \prod_{\ell=1}^{l} 2^{|\sigma_{\ell}|-1}\mathrm{Tr}(\prod_{i \in \sigma_{\ell}} R_i(\Tilde{M}(\tau_i)/\sqrt{n})),
    \end{split}
\end{equation*}
where 
\begin{equation*}
    R_i(x)=\begin{cases} P_{2k_i-1}(x) \text{ if } i\leq j_1\\ P_{2k_{i-j_1}'}(x) \text{ if } i>j_1 \end{cases}, \hspace{1cm} \tau_i=\begin{cases} t_i \text{ if } i\leq j_1\\ t_{i-j_1}'(x) \text{ if } i>j_1 \end{cases}. 
\end{equation*} 

Next, we can use the factorization property, Corollary~\ref{Factorization Lemma}, to see that any permutation $\sigma$ with $k$ fixed points at $i_1,\dots ,i_k$ contributes at leading order. We have
\begin{equation*}
    \begin{split}
        n^{-(j_1+j_2)}&\prod_{i \in [j_1+j_2]} e^{-|\tau_{i}-\tau_{\sigma(i)}|} \prod_{l=1}^{k}\E[\mathrm{Tr}(R_{i_l}(\tau_{i_l}))] \prod_{|\sigma_{\ell}|>1} 2^{|\sigma_{\ell}|-1}\E\bigg[\mathrm{Tr}(\prod_{i \in \sigma_{\ell}} R_i(\Tilde{M}(\tau_i)/\sqrt{n}))\bigg]\\
        &=O(n^{-(j_1+j_2)} 1^{l} n^{(j_1+j_2-l)/2}) \\
        &= O(n^{-(j_1+j_2+l)/2}),
    \end{split}
\end{equation*}
where we make use of the fact that $\E[\mathrm{Tr}(P(\Tilde{M}(t)/\sqrt{n}))]=O(n)$ for any polynomial $P(x)$, but if $P(x)=R_i(x)$, then $\E[\mathrm{Tr}(P(\Tilde{M}(t)/\sqrt{n}))]=O(1)$ by Lemma \ref{Semicircular Cov}. However, any term coming from a permutation with less than $j/2$ number of cycles will also yield an $O(n^{-(j-1)/2})$ contribution from the same argument. Thus the contributing terms must come from pairings $\sigma \in \mathcal{P}_2([j_1+j_2])$.

Finally, if $\sigma$ pairs an index from $[j_1]$ with another from $[j_1+1:j_2]$, then we also get an order $O(n^{-j/2-1})$ contribution since for any $i,i'$, $\E[\mathrm{Tr}(P_{2k_i}(\Tilde{M}(t_i)/\sqrt{n}))P_{2k'_{i'}-1}(\Tilde{M}(t_{i'})/\sqrt{n}))]=O(1)$, once again by Lemma~\ref{Semicircular Cov}. Thus the contributing permutation is $\sigma=\sigma_1\sigma_2$ where $\sigma_1 \in \mathcal{P}_2([j_1])$ and $\sigma_2 \in \mathcal{P}_2([j_1+1:j_2])$. Thus, we get
\begin{equation*}
    \begin{split}
        \E\bigg[ &\bigg( \prod_{i=1}^{j_1} (\theta(t_i)^T P_{2k_i-1}(\Tilde{M}(t_i)/\sqrt{n})\theta(t_i))\bigg)\bigg( \prod_{i=1}^{j_1} (\theta(t_i')^T P_{2k_i'}(\Tilde{M}(t_i')/\sqrt{n})\theta(t_i))\bigg)\bigg] \\
        &=n^{-(j_1+j_2)/2}\left(\bigg(2^{j_1/2}\sum_{\sigma \in \mathcal{P}_2(j_1)} \prod_{(i_1,i_2) \in \pi} e^{-(2k_{i_1}+1)|t_{i_1}-t_{i_2}|}\delta_{k_{i_1},k_{i_2}} \bigg)\right. \\
        &\hspace{3.5cm}\times \left.\bigg(2^{j_2/2}\sum_{\sigma \in \mathcal{P}_2(j_2)} \prod_{(i_1,i_2) \in \pi}e^{-2(k'_{i_1}+1)|t_{i_1}'-t_{i_2}'|}\delta_{k'_{i_1},k'_{i_2}} \bigg)+O(n^{-1/2})\right).
    \end{split}
\end{equation*}
Multiplying both sides by $n^{(j_1+j_2)/2}$ completes the proof of Proposition \ref{Convergence of Moments}. 
\end{proof}

Finally, we show tightness of the process $E_n(t)$, which we prove after Lemma~\ref{Tightness Lemma}.  
\begin{prop}\label{lem:tightness}
    $E_n(t)$ is a tight sequence of random $C[0,T]$ curves. 
\end{prop}

\appendix
\section{Proof of Lemma~\ref{Concentration Bound}}\label{Concentration Proof}

\begin{lm}
    Define 
    \begin{equation*}
        G(A,x):=F(A,P_0(A)x,P_1(A)x,\dots,P_k(A) x),
    \end{equation*}
    and also 
    \begin{equation*}
        \Tilde{G}(A,x) := G\left(\frac{\min(100,||A||)}{||A||}A,\frac{\min(100,||x||)}{||x||}x\right). 
    \end{equation*}
    Then $\Tilde{G}(A,x)$ is Lipschitz on $\R^{n^2+n}$ with $||\Tilde{G}||_{Lip} \lesssim 1$.
\end{lm}
\begin{proof}
Let $\mathcal{A}=\{(A,x)\}\in \R^{n^2+n}: ||x|| \leq 100, ||A||_{op} \leq 100\}$. Then, on the set $\mathcal{A}$ we have
\begin{equation*}
    \begin{split}
        ||P_k&((A+\Delta A)(x+\Delta x))-P_k(A x)|| \\
        &\leq ||(P_k(A+\Delta A)-P_k(A))x||+||P_k(A+\Delta A)\Delta x|| \\
        &\leq ||P_k(A+\Delta A)-P_k(A)|| \cdot ||x|| + ||P_k(A+\Delta A)||\cdot ||\Delta x|| \\
        &\leq \sum_{\ell=0}^{k}\left(\sum_{j=0}^{\ell-1} c_{j,\ell}||A||^j ||\Delta A||^{\ell-j} ||x||+\sum_{j=0}^{\ell} c_{j,\ell}||A||^j ||\Delta A||^{\ell-j} ||\Delta x|| \right)\\
        &\lesssim ||\Delta A|| + ||\Delta x||.
    \end{split}
\end{equation*}
For some coefficients $c_{j, l}$. Recalling that $F$ is uniformly Lipschitz on compact sets. By the previous computation and the gradient assumption on $F$, $G$ is Lipschitz on $\mathcal{A}$. Thus, $\Tilde{G}$ is continuous and agrees with $G$ on $\mathcal{A}$, and is obviously Lipschitz with the same constant on all of $\mathcal{A}^c$. 
\end{proof}

\begin{lm}\label{Supremum Lemma}
    For some constants $C,c(T)>0$. We have the following sub-Gaussianity estimates. For any $u,\epsilon>0$ we have
    \begin{gather}\label{Norm Supremum}
        \pr\bigg(\sup_{t \in [0,T]}(||\theta(t)||-1)>\frac{u}{\sqrt{n}}\bigg) \leq  Ce^{-c(T)\frac{u^2}{\log(n)}},\\
    \label{Operator Norm Supremum} \pr\bigg(\sup_{t \in [0,T]}||M(t)||_{op}-(2+\epsilon)\sqrt{n}>u\bigg) \leq Ce^{-c(T)\frac{u^2}{\log(n)}},\\ \label{Tilde G Concentration}
    \pr\bigg(\sup_{t\in [0,T]}|\Tilde{G}(M(t)/\sqrt{n},\theta(t))-\E[\Tilde{G}(M(t)/\sqrt{n},\theta(t))]|>u\bigg) \leq Ce^{-c(T)\frac{nu^2}{\log(n)}}.
    \end{gather}
\end{lm}

\begin{proof}

Observe that $X(t)-X(s) \sim N(0,|t-s|I)$,  $\E||X(t)-X(s)||\leq (\E||X(t)-X(s)||^2)^{1/2}=\sqrt{n|t-s|}$ and also the function $x\mapsto ||x||$ is $1$-Lipschitz. Thus by the Gaussian concentration \cite{Ha16} we have
\begin{equation*}
    \pr(||X(t)-X(s)||>\sqrt{n|t-s|}+u)\leq \pr(||X(t)-X(s)||-\E||X(t)-X(s)||>u) \leq e^{-\frac{u^2}{4|t-s|}},
\end{equation*}
for any $u>0$. So in an interval $I \subset [0,T]$ of length $\leq n^{-1}$, we have $\pr(||X(t)-X(s)||>1+u)\leq e^{-\frac{u^2}{4|t-s|}}$ for all $t,s \in I$. In particular, it follows for $u>0$ that
\begin{equation}\label{Norm Continuity}
    \pr(||X(t)-X(s)||>u)\leq C e^{-u^2/4},
\end{equation}
for some $C>0$. Similarly, Gaussian concentration also implies the one time estimate
\begin{equation*}
    \pr(||X(t)||-\sqrt{n}>u)\leq 2 e^{-\frac{u^2}{4}}.
\end{equation*}
Thus $||X(t)||-\sqrt{n}$ is a separable Sub-Gaussian process on $(I,d)$ with $d(t,s)=\sqrt{|t-s|}$. From the chaining tail inequality \cite[Thm 5.29]{Ha16} we obtain
\begin{equation}
    \pr\bigg(\sup_{t \in I} ||X(t)||-\sqrt{n}>u\bigg) \leq Ce^{-c u^2}, 
\end{equation}
on $I$. 

Subdividing $[0,T] = \cup_{j=1}^{Cn} I_j$ for intervals $I_j$ of length $\leq 1/n$, we recall the max of $Cn$ sub-Gaussian random variables is $C\sqrt{\log(n)}$ sub-Gaussian, in which
\begin{equation*}
    \pr\bigg(\sup_{t \in [0,T]} ||X(t)||-\sqrt{n}>u\bigg) =\pr\bigg(\max_{i=1,\dots,Cn}\sup_{t \in I_j} ||X(t)||-\sqrt{n}>u\bigg)\leq Ce^{-c \frac{u^2}{\log(n)}},
\end{equation*}
holds for any $u>0$ and some constant $C,c$ which only depends on $T$. This proves \eqref{Norm Supremum} as $\theta(t)=\frac{X(t)}{\sqrt{n}}$.

The proof of the second inequality, \eqref{Operator Norm Supremum}, is identical once we recall that for $n$ sufficiently large $\E[||M(t)||_{op}] \leq (2+\epsilon)\sqrt{n}$ by the Bai-Yin theorem (i.e. see \cite{Anderson_Guionnet_Zeitouni_2009}). Note along the way we obtain,
\begin{equation}\label{Operator Norm Continuity}
    \pr(||M(t)-M(s)||_{op}>\sqrt{(2+\epsilon)n|t-s|}+u)\leq C e^{-\frac{cu^2}{4|t-s|}}.
\end{equation}

Finally for \eqref{Tilde G Concentration}, by the Gaussian concentration and the Lipschitz bound on $\Tilde{G}$, we have the one time estimate
\begin{equation*}
    \pr\left(\Tilde{G}(M(t)/\sqrt{n},\theta(t))-\E\Tilde{G}(M(t)/\sqrt{n},\theta(t))>u\right)\leq e^{-c n u^2}.
\end{equation*}
Next, for an interval $I \subset [0,T]$ of length $\leq n^{-1}$, combining the Lipschitz bound on $\Tilde{G}$, \eqref{Norm Continuity} and \eqref{Operator Norm Continuity}, we have 
\begin{equation}\label{two time tildeG estimate}
    \pr\left(|\Tilde{G}(M(t)/\sqrt{n},\theta(t))-\Tilde{G}(M(s)/\sqrt{n},\theta(s))|>u\right)\leq C e^{-c n u^2},
\end{equation}
for $t,s \in I$. Now proceeding identically to the proof of $\eqref{Norm Supremum}$, we obtain \eqref{Tilde G Concentration}.
\end{proof}

\begin{lm}
We have
    \begin{equation}\label{G almost always tilde G}
        \pr\bigg(\sup_{t \in [0,T]}|\Tilde{G}(M(t)/\sqrt{n},\theta(t))-G(M(t)/\sqrt{n},\theta(t))| \neq 0\bigg) \lesssim e^{-c\frac{n}{\log(n)}}.
    \end{equation}
\end{lm}
\begin{proof}
Notice that the event $\{\sup_{t \in [0,T]}|\Tilde{G}(M(t)/\sqrt{n},\theta(t))-G(M(t)/\sqrt{n},\theta(t))| \neq 0\}$ is a subset of $\{\sup_{t \in [0,T]}||\theta(t)||>100\}\cup \{\sup_{t \in [0,T]}||M(t)||_{op}>100\sqrt{n}\}$. Thus the result follows from \eqref{Norm Supremum} and \eqref{Operator Norm Supremum}.
\end{proof}

\begin{lm}
We have 
    \begin{equation}\label{Expectation G almost Expectation  tilde G}
        \sup_{t \in [0,T]}|\E\Tilde{G}(M(t)/\sqrt{n},\theta(t))-\E G(M(t)/\sqrt{n},\theta(t))| \lesssim e^{-c\frac{n}{\log(n)}},
    \end{equation}
    for sufficiently large $n$. 
\end{lm}
\begin{proof}
    First we can get rid of the supremum by stationarity of $M(t)$ and $\theta(t)$. Next, by the assumption $|F(x)| \leq C(1+||x||)^d$, we immediately know that $\E|\Tilde{G}(M(t)/\sqrt{n},\theta(t))-G(M(t)/\sqrt{n},\theta(t))|^2 \leq Cn^{c(d)}$. Then by applying Cauchy-Schwartz inequality and \eqref{G almost always tilde G} we get
    \begin{equation*}
        \begin{split}
             \E|\Tilde{G}&(M(t)/\sqrt{n},\theta(t))-G(M(t)/\sqrt{n},\theta(t))| \\
            &=\E\left[|\Tilde{G}(M(t)/\sqrt{n},\theta(t))-G(M(t)/\sqrt{n},\theta(t))| \cdot \mathbbm{1}_{|\Tilde{G}(M(t)/\sqrt{n},\theta(t))-G(M(t)/\sqrt{n},\theta(t))| \neq 0}\right]\\
            &\leq \E\left[|\Tilde{G}(M(t)/\sqrt{n},\theta(t))-G(M(t)/\sqrt{n},\theta(t))|^2\right]^{1/2}\\
            & \hspace{3cm}\times\pr\bigg(\sup_{t \in [0,T]}|\Tilde{G}(M(t)/\sqrt{n},\theta(t))-G(M(t)/\sqrt{n},\theta(t))| \neq 0\bigg)^{1/2}\\
            &\leq C n^{c(d)} e^{-c(T,k)\frac{n}{\log(n)}} \\
            & \leq C e^{-c(T,k,d)\frac{n}{\log(n)}},
        \end{split}
    \end{equation*}
    for sufficiently large $n$. This ends the proof. 
\end{proof}
Lemma~\ref{Concentration Bound} now follows by combining \eqref{Tilde G Concentration}, \eqref{G almost always tilde G}, and \eqref{Expectation G almost Expectation  tilde G} through the triangle inequality. Lastly we record a quantitative continuity estimates which is used to prove tightness that is already introduced in Proposition~\ref{lem:tightness}. 
\begin{lm}\label{Tightness Lemma}
    Fixing a $\delta>0$, for $n$ large enough we have
    \begin{equation*}
        \pr\bigg(\sup_{t,s \in [0,T]}\bigg|\frac{G(M(t)/\sqrt{n},\theta(t))-G(M(s)/\sqrt{n},\theta(s))}{|t-s|^{1/2-\delta}}\bigg|>u\bigg) \leq Ce^{-c \frac{u^2}{\log(n)}}+Ce^{-c\frac{n}{\log(n)}},
    \end{equation*}
    where $C,c$ may depend on $T$ and $\delta$.
\end{lm}
\begin{proof}
    From Lemma~\ref{G almost always tilde G}, we have $G(M(t)/\sqrt{n},\theta(t)) =\Tilde{G}(M(t)/\sqrt{n},\theta(t))$ on all of $[0,T]$. Thus it suffices to prove
    \begin{equation*}
        \pr\bigg(\sup_{t,s \in [0,T]}\bigg|\frac{\Tilde{G}(M(t)/\sqrt{n},\theta(t))-\Tilde{G}(M(s)/\sqrt{n},\theta(s))}{|t-s|^{1/2-\delta}}\bigg|>u\bigg) \leq Ce^{-c u^2}.
    \end{equation*}
    Towards this goal, we use a classic result in the theory of Brownian motion \cite{gall2016brownian}, that 
    \begin{equation*}
        \pr\bigg(\sup_{t,s \in [0,T]}\frac{||\theta(t)-\theta(s)||}{|t-s|^{1/2-\delta}}>u \bigg)< Ce^{-cnu^2}.
    \end{equation*}

    Next recall we previously obtained the estimate
    \begin{equation*}
        \pr(||M(t)-M(s)||_{op}> \sqrt{(2+\epsilon)|t-s|n}+u)< Ce^{-cu^2},
    \end{equation*}
    and thus 
    \begin{equation*}
        \pr\bigg(\frac{|M(t)-M(s)|}{|t-s|^{1/2-\delta}}> \sqrt{(2+\epsilon)|t-s|^{\delta}n}+u|t-s|^{\delta-1/2}\bigg)< Ce^{-cu^2}. 
    \end{equation*}
    So on an interval $I$ of size $\leq n^{-1/\delta}$, we have
    \begin{equation*}
        \pr\bigg(\frac{|M(t)-M(s)|}{|t-s|^{1/2-\delta}}>u \bigg)< Ce^{-cu^2}.
    \end{equation*}
    Taking a union of $n^{1/\delta}$ such intervals and recalling once again properties for the maximum of sub-Gaussian random variables, we have
    \begin{equation*}
        \pr\bigg(\sup_{t,s \in I}\frac{||M(t)-M(s)||_{op}}{\sqrt{n}|t-s|^{1/2-\delta}}>u \bigg)< Ce^{-cn\frac{u^2}{\log(n)}}.
    \end{equation*}
    The Lemma now follows from Lipschitz continuity of $\Tilde{G}$.
\end{proof}
We finally prove Proposition~\ref{lem:tightness}. 

\begin{proof}[Proof of Proposition~\ref{lem:tightness}]
    By Lemma~\ref{Tightness Lemma} with $G(A,x)= x^T P_{2j-1}(A)x$, and $G(A,x)= x^T P_{2j}(A)x$ for $j\leq k-1$, for any $\epsilon, \delta>0$ there is ball $B \in C^{1/2-\delta}$ such that $\pr(E_n \in B)>1-\epsilon$ for each $n$, and recalling the compact embeddings of Holder spaces, $\overline{B}$ embeds compactly in $C[0,T]$, thus $\{\Tilde{E}_n\}$ is tight in  $C[0,T]$.
\end{proof}

\section*{Acknowledgements}

{\small  
This material is based upon work supported by the U.S. National Science Foundation under award Nos CNS-2346520, PHY-2028125, RISE-2425761, DMS-2325184, OAC-2103804, and OSI-2029670, by the Defense Advanced Research Projects Agency (DARPA) under Agreement No. HR00112490488, by the Department of Energy, National Nuclear Security Administration under Award Number DE-NA0003965 and by the United States Air Force Research Laboratory under Cooperative Agreement Number FA8750-19-2-1000. Neither the United States Government nor any agency thereof, nor any of their employees, makes any warranty, express or implied, or assumes any legal liability or responsibility for the accuracy, completeness, or usefulness of any information, apparatus, product, or process disclosed, or represents that its use would not infringe privately owned rights. Reference herein to any specific commercial product, process, or service by trade name, trademark, manufacturer, or otherwise does not necessarily constitute or imply its endorsement, recommendation, or favoring by the United States Government or any agency thereof. The views and opinions of authors expressed herein do not necessarily state or reflect those of the United States Government or any agency thereof." The views and conclusions contained in this document are those of the authors and should not be interpreted as representing the official policies, either expressed or implied, of the United States Air Force or the U.S. Government. 

\noindent R. N. is supported by the NSF GRFP-2141064. }

\bibliographystyle{alpha}
\bibliography{refs}

@book{abramowitz1948handbook,
  title={Handbook of Mathematical Functions with Formulas, Graphs, and Mathematical Tables},
  author={Abramowitz, Milton and Stegun, Irene A},
  volume={55},
  year={1948},
  publisher={US Government printing office}
}

@article{AlGu13,
author = {Romain Allez and Alice Guionnet},
title = {{A Diffusive Matrix Model for Invariant $\beta$-Ensembles}},
volume = {18},
journal = {Electronic Journal of Probability},
publisher = {Institute of Mathematical Statistics and Bernoulli Society},
pages = {1 -- 30},
year = {2013},
doi = {10.1214/EJP.v18-2073},
URL = {https://doi.org/10.1214/EJP.v18-2073}
}

@article{allez2014tracy,
  title={Tracy-{W}idom at High Temperature},
  author={Allez, Romain and Dumaz, Laure},
  journal={Journal of Statistical Physics},
  volume={156},
  pages={1146--1183},
  year={2014},
  publisher={Springer}
}

@book{Anderson_Guionnet_Zeitouni_2009, place={Cambridge}, series={Cambridge Studies in Advanced Mathematics}, title={An Introduction to Random Matrices}, publisher={Cambridge University Press}, author={Anderson, Greg W. and Guionnet, Alice and Zeitouni, Ofer}, year={2009}, collection={Cambridge Studies in Advanced Mathematics}}

@article{BoFe08,
   title={The {A}iry$_1$ Process is not the Limit of the Largest Eigenvalue in {GOE} Matrix Diffusion},
   volume={133},
   ISSN={1572-9613},
   url={https://doi.org/10.1007/s10955-008-9621-0},
   DOI={10.1007/s10955-008-9621-0},
   number={3},
   journal={Journal of Statistical Physics},
   author={Bornemann, Folkmar and Ferrari, Patrik L. and Prähofer, Michael},
   year={2008}, pages={405-415} }

@article{bauer1959sequential,
  title={Sequential Reduction to Tridiagonal Form},
  author={Bauer, FL},
  journal={Journal of the Society for Industrial and Applied Mathematics},
  volume={7},
  number={1},
  pages={107--113},
  year={1959},
  publisher={SIAM}
}

@book{deift2000orthogonal,
  title={Orthogonal Polynomials and Random Matrices: A Riemann-Hilbert Approach: A Riemann-Hilbert Approach},
  author={Deift, Percy},
  volume={3},
  year={2000},
  publisher={American Mathematical Soc.}
}

@article{DuEd02,
   title={Matrix Models for Beta Ensembles},
   volume={43},
   ISSN={1089-7658},
   url={http://dx.doi.org/10.1063/1.1507823},
   DOI={10.1063/1.1507823},
   number={11},
   journal={Journal of Mathematical Physics},
   publisher={AIP Publishing},
   author={Dumitriu, Ioana and Edelman, Alan},
   year={2002},
   month=nov, pages={5830–5847} }

@article{dumitriu2006global,
  title={Global Spectrum Fluctuations for the $\beta$-{H}ermite and $\beta$-{L}aguerre Ensembles Via Matrix Models},
  author={Dumitriu, Ioana and Edelman, Alan},
  journal={Journal of Mathematical Physics},
  volume={47},
  number={6},
  year={2006},
  publisher={AIP Publishing}
}

@article{dumitriu2012global,
  title={Global Fluctuations for Linear Statistics of $\beta$-{J}acobi Ensembles},
  author={Dumitriu, Ioana and Paquette, Elliot},
  journal={Random Matrices: Theory and Applications},
  volume={1},
  number={04},
  pages={1250013},
  year={2012},
  publisher={World Scientific}
}

@article{Dy62,
  title={A {B}rownian-Motion model for the Eigenvalues of a Random Matrix},
  author={Dyson, Freeman J},
  journal={Journal of Mathematical Physics},
  volume={3},
  number={6},
  pages={1191--1198},
  year={1962},
  publisher={American Institute of Physics}
}

@misc{Ed04,
title= {Advances in Random Matrix Theory: Let there be Tools},
author = {Alan Edelman},
year = {2004},
note= {Presented at World Congress, Bernoulli Society, Barcelona, Spain},
pages={slides 30-32},
URL= {https://math.mit.edu/~edelman/talks/2004/tools.ppt},
}

@article{EdSu07,
   title={From Random Matrices to Stochastic Operators},
   volume={127},
   ISSN={1572-9613},
   url={https://doi.org/10.1007/s10955-006-9226-4},
   DOI={10.1007/s10955-006-9226-4},
   number={6},
   journal={Journal of Statistical Physics},
   author={Edelman, Alan and Sutton, Brian D.},
   year={2007}, pages={1121-1165} }

@book{gall2016brownian,
  title={Brownian Motion, {M}artingales, and Stochastic Calculus},
  author={Gall, J.F.L.},
  isbn={9783319310893},
  series={Graduate Texts in Mathematics},
  url={https://books.google.com/books?id=G00WDAAAQBAJ},
  year={2016},
  publisher={Springer International Publishing}
}

@article{gorin2018stochastic,
  title={Stochastic {A}iry Semigroup Through Tridiagonal Matrices},
  author={Gorin, Vadim and Shkolnikov, Mykhaylo},
  journal={The Annals of Probability},
  volume={46},
  number={4},
  pages={2287--2344},
  year={2018},
  publisher={JSTOR}
}

@misc{Ha16,
      title={Probability in High Dimension}, 
      author={van Handel, Ramon},
      year={2016},
      url={https://web.math.princeton.edu/~rvan/APC550.pdf}, 
}

@misc{HoPa17,
      title={Tridiagonal Models for {D}yson {B}rownian Motion}, 
      author={Diane Holcomb and Elliot Paquette},
      year={2017},
      eprint={1707.02700},
      archivePrefix={arXiv},
      primaryClass={math.PR},
      url={https://arxiv.org/abs/1707.02700}, 
}

@article{krishnapur2016universality,
  title={Universality of the Stochastic {A}iry operator},
  author={Krishnapur, Manjunath and Rider, Brian and Vir{\'a}g, B{\'a}lint},
  journal={Communications on Pure and Applied Mathematics},
  volume={69},
  number={1},
  pages={145--199},
  year={2016},
  publisher={Wiley Online Library}
}

@article{LambertPaquette03,
  title={Strong approximation of {G}aussian $\beta$-ensemble characteristic polynomials: the edge regime and the stochastic {A}iry function},
  author={Lambert, Gaultier and Paquette, Elliot},
  journal={arXiv preprint arXiv:2009.05003},
  year={2020}
}

@article{li2010beta,
  title={The Beta-{H}ermite and Beta-{L}aguerre Processes},
  author={Li, Luen-Chau},
  journal={arXiv preprint arXiv:1007.3905},
  year={2010}
}

@book{Nica_Speicher_2006, place={Cambridge}, series={London Mathematical Society Lecture Note Series}, title={Lectures on the Combinatorics of Free Probability}, publisher={Cambridge University Press}, author={Nica, Alexandru and Speicher, Roland}, year={2006}, collection={London Mathematical Society Lecture Note Series}}

@article{ramirez2011beta,
  title={Beta Ensembles, Stochastic {A}iry Spectrum, and a Diffusion},
  author={Ramirez, Jose and Rider, Brian and Vir{\'a}g, B{\'a}lint},
  journal={Journal of the American Mathematical Society},
  volume={24},
  number={4},
  pages={919--944},
  year={2011}
}

@article{wilkinson1962householder,
  title={Householder's Method for Symmetric Matrices},
  author={Wilkinson, JH},
  journal={Numerische Mathematik},
  volume={4},
  pages={354--361},
  year={1962},
  publisher={Springer}
}

@article{edelman2005random,
  title={Random Matrix Theory},
  author={Edelman, Alan and Rao, N Raj},
  journal={Acta numerica},
  volume={14},
  pages={233--297},
  year={2005},
  publisher={Cambridge University Press}
}

@inproceedings{reuther2018interactive,
title={Interactive Supercomputing on 40,000 Cores for Machine Learning and Data Analysis},
author={Reuther, Albert and Kepner, Jeremy and Byun, Chansup and Samsi, Siddharth and Arcand,
William and Bestor, David and Bergeron, Bill and Gadepally, Vijay and Houle, Michael and Hubbell,
Matthew and Jones, Michael and Klein, Anna and Milechin, Lauren and Mullen, Julia and Prout,
Andrew and Rosa, Antonio and Yee, Charles and Michaleas, Peter},
booktitle={2018 IEEE High Performance extreme Computing Conference (HPEC)},
pages={1–6},
year={2018},
organization={IEEE}
}

@article{trogdon2024computing,
  title={Computing the {T}racy-{W}idom Distribution for Arbitrary $\beta>0$},
  author={Trogdon, Thomas and Zhang, Yiting},
  journal={SIGMA. Symmetry, Integrability and Geometry: Methods and Applications},
  volume={20},
  pages={005},
  year={2024},
  publisher={SIGMA. Symmetry, Integrability and Geometry: Methods and Applications}
}

@article{valko2017sine,
  title={The Sine$_{\beta}$ Operator},
  author={Valk{\'o}, Benedek and Vir{\'a}g, B{\'a}lint},
  journal={Inventiones mathematicae},
  volume={209},
  pages={275--327},
  year={2017},
  publisher={Springer}
}

@book{forrester2010log,
  title={Log-Gases and Random Matrices (LMS-34)},
  author={Forrester, Peter J},
  year={2010},
  publisher={Princeton university press}
}

@article{deift2021conjugate,
  title={The Conjugate Gradient Algorithm on Well-Conditioned {W}ishart Matrices is Almost Deterministic},
  author={Deift, Percy and Trogdon, Thomas},
  journal={Quarterly of Applied Mathematics},
  volume={79},
  number={1},
  pages={125--161},
  year={2021}
}

@article{paquette2023universality,
  title={Universality for the Conjugate Gradient and {MINRES} Algorithms on Sample Covariance Matrices},
  author={Paquette, Elliot and Trogdon, Thomas},
  journal={Communications on Pure and Applied Mathematics},
  volume={76},
  number={5},
  pages={1085--1136},
  year={2023},
  publisher={Wiley Online Library}
}

@article{ding2022conjugate,
  title={The Conjugate Gradient Algorithm on a General Class of Spiked Covariance Matrices},
  author={Ding, Xiucai and Trogodon, Thomas},
  journal={Quarterly of applied mathematics},
  year={2022}
}

@article{gorin2024airy,
  title={Airy$_\beta$ Line Ensemble and its {L}aplace Transform},
  author={Gorin, Vadim and Xu, Jiaming and Zhang, Lingfu},
  journal={arXiv preprint arXiv:2411.10829},
  year={2024}
}

@article{dunkl2024eigenvalues,
  title={Eigenvalues of {H}eckman-{P}olychronakos Operators},
  author={Dunkl, Charles and Gorin, Vadim},
  journal={arXiv preprint arXiv:2412.01938},
  year={2024}
}

@article{huang2024convergence,
  title={A Convergence Framework for {A}iry$_\beta $ Line Ensemble Via Pole Evolution},
  author={Huang, Jiaoyang and Zhang, Lingfu},
  journal={arXiv preprint arXiv:2411.10586},
  year={2024}
}

@article {PierrePerturbations,
    AUTHOR = {Touzo, L\'eo and Le Doussal, Pierre and Schehr, Gr\'egory},
     TITLE = {Fluctuations in the active {D}yson {B}rownian motion and the
              overdamped {C}alogero-{M}oser model},
   JOURNAL = {Phys. Rev. E},
  FJOURNAL = {Physical Review E},
    VOLUME = {109},
      YEAR = {2024},
    NUMBER = {1},
     PAGES = {Paper No. 014136, 25},
      ISSN = {2470-0045,2470-0053},
   MRCLASS = {82C22},
  MRNUMBER = {4707743},
       DOI = {10.1103/physreve.109.014136},
       URL = {https://doi.org/10.1103/physreve.109.014136},
}

@article {GorinPerturbation,
    AUTHOR = {Gorin, Vadim and Kleptsyn, Victor},
     TITLE = {Universal objects of the infinite beta random matrix theory},
   JOURNAL = {J. Eur. Math. Soc. (JEMS)},
  FJOURNAL = {Journal of the European Mathematical Society (JEMS)},
    VOLUME = {26},
      YEAR = {2024},
    NUMBER = {9},
     PAGES = {3429--3496},
      ISSN = {1435-9855,1435-9863},
   MRCLASS = {60B20 (15B33 15B52 33C45 81T32)},
  MRNUMBER = {4767497},
MRREVIEWER = {Ramon\ van Handel},
       DOI = {10.4171/jems/1336},
       URL = {https://doi.org/10.4171/jems/1336},
}

@misc{pierreprivatecomm,
  author = {Le Doussal, Pierre},
  year = {2025},
  howpublished = {Private communication}
}

\end{document}